\documentclass[11pt,a4paper,oneside]{article}
\usepackage[english]{babel}
\usepackage[T1]{fontenc}
\usepackage{amsmath}
\usepackage{amsthm}
\usepackage{amsfonts}
\usepackage{amssymb}
\usepackage{lmodern} 
\usepackage{indentfirst}		
\usepackage{microtype} 			
\usepackage[numbers]{natbib}
\usepackage{hyphenat}
\usepackage{pgfplots}  	
\usepackage[latin1]{inputenc}
\usepackage{indentfirst}

\definecolor{blackgreen}{RGB}{0,80,0}

\usepackage[plainpages=false,pdfpagelabels,pdfusetitle,pdfdisplaydoctitle, pdfstartpage=1, pdfstartview={{XYZ null null 1.00}}]{hyperref}
\hypersetup{
	pdftitle={critical nom concentration},    
	pdfauthor={Mykael Cardoso, Luiz Gustavo Farah and Carlos Guzmán},    
	pdfnewwindow=true,      
	colorlinks=true,      
	linkcolor=blackgreen,         
	citecolor=red}

\usepackage[top=2cm, bottom=2cm, left=2cm, right=2cm]{geometry}
\usepackage{mathtools}
\usepackage{mathabx}
\mathtoolsset{showonlyrefs}

\newcommand{\Real}{\mathbb R}
\newtheorem{thm}{Theorem}[section]
\newtheorem{prop}[thm]{Proposition}
\newtheorem{lemma}[thm]{Lemma}
\newtheorem{rem}[thm]{Remark}

\newtheorem{coro}[thm]{Corollary}
\numberwithin{equation}{section}

\title{On well-posedness and concentration of blow-up solutions for the intercritical inhomogeneous NLS equation}
\author{Mykael Cardoso, Luiz Gustavo Farah and Carlos M. Guzm\'an} 
\date{} 

\begin{document}
\maketitle
	
\begin{abstract}\noindent
We consider the focusing inhomogeneous nonlinear Schr\"odinger (INLS) equation in $\mathbb{R}^N$
$$i \partial_t u +\Delta u + |x|^{-b} |u|^{2\sigma}u = 0,$$
where $N\geq 2$ and $\sigma$, $b>0$. We first obtain a small data global result in $H^1$, which, in the two spatial dimensional case, improves the third author result in \cite{Boa} on the range of $b$. For $N\geq 3$ and $\frac{2-b}{N}<\sigma<\frac{2-b}{N-2}$, we also study the local well posedness in $\dot H^{s_c}\cap \dot H^1 $, where $s_c=\frac{N}{2}-\frac{2-b}{2\sigma}$. Sufficient conditions for global existence of solutions in $\dot H^{s_c}\cap \dot H^1$ are also established, using a  Gagliardo-Nirenberg type estimate. Finally, we study the $L^{\sigma_c}-$norm concentration phenomenon, where $\sigma_c=\frac{2N\sigma}{2-b}$, for finite time blow-up solutions in $\dot H^{s_c}\cap \dot H^1$ with bounded $\dot H^{s_c}-$norm. Our approach is based on the compact embedding of $\dot H^{s_c}\cap \dot H^1$ into a weighted $L^{2\sigma+2}$ space.

\end{abstract}

\section{Introduction}
In this paper, we study the initial value problem (IVP) for the inhomogeneous nonlinear Schr\"odinger (INLS) equation
\begin{equation}
\begin{cases}
i \partial_t u + \Delta u + |x|^{-b} |u|^{2 \sigma}u = 0, \,\,\, x \in \mathbb{R}^N, \,t>0,\\
u(x,0) = u_0(x) ,
\end{cases}
\label{PVI}
\end{equation}
where $\sigma, b>0$ are real numbers. 


The INLS model is a extension of the classical nonlinear Schr\"odinger equation (case $b=0$), extensively studied in recent years (see, \citet{Sulem} \citet{Bo99}, \citet{cazenave}, \citet{LiPo15}, \citet{TaoBook}, \citet{Fi15} and the references therein). As suggested by \citet{Gill} and \citet{LIu}, it can be used as a model for the propagation of laser beams in nonlinear optics. 

The well-posedness for the INLS equation $H^1(\Real^{N})$ was first studied by \citet{g_8}. Specifically, they showed that \eqref{PVI} is locally well-posed in $H^1(\Real^{N})$ if $0 < \sigma < \frac{2-b}{N-2}$ ($0 < \sigma < \infty$, if $N=1,2$) and $0 < b < 2$. The $H^1$ flow admits the conservation of mass $M[u]$ and energy $E[u]$ defined by
\begin{equation}\label{mass}
M\left[u(t) \right] = \int |u(t)|^2 dx = M[u_0],
\end{equation}
\begin{equation}\label{energy}
E\left[u(t) \right] = \frac{1}{2}\int |\nabla u(t)|^2 dx - \frac{1}{2 \sigma+2} \int |x|^{-b}|u(t)|^{2 \sigma+2} dx = E[u_0].
\end{equation}

An important symmetry is given by the scaling 
\begin{align}
u_\rho(x,t)=\rho^{\frac{2-b}{2\sigma}}u(\rho x,\rho^2t),\quad \rho>0.
\end{align}
It is easy to see that if $u(x, t)$ is a solution to the INLS equation, then $u_\rho(x,t)$ is also a solution.

The critical Sobolev index $s_c$ related to \eqref{PVI} is such that the homogeneous Sobolev space $H^{s_c}(\Real^{N})$ leaves the scaling symmetry invariant, explicitly 
$$
s_c=\frac{N}{2}-\frac{2-b}{2\sigma}.
$$
We say that the problem is mass-critical (or $L^2$-critical) if $s_c = 0$, energy-critical (or $\dot{H}^1$-critical) if $s_c=1$ and intercritical if $0<s_c<1$. In terms of $\sigma$ and $b$ we can reformulate these condition as
\begin{itemize}
\item[] Mass-critical: $\sigma = \frac{2-b}{N}$.;
\item[] Energy-critical\footnote{Note that this case is only possible if $N\geq3$.}: $\sigma = \frac{2-b}{N-2}$;
\item[] Intercritical: $\frac{2-b}{N}<\sigma<\frac{2-b}{N-2}$ ($\frac{2-b}{N}<\sigma<\infty$, if $N=1,2$).
\end{itemize}


Several other authors studied the well-posedness of the IVP \eqref{PVI}. Genoud \cite{genoud2012critical} studied the mass-critical case and showed the global well-posedness in $H^1(\mathbb{R}^N)$, $N\geq 1$, provided that the mass of the initial data staisfies an appropriate smallness condition. This result was extended to the intercritical case by the second author in \cite{Farah_well}. Applying a different method, based on the Strichartz estimates satisfied by the linear evolution, the third author in \cite{Boa} established the local well-posedness in $H^1(\Real^N)$, for $N \geq 2$, $0 < \sigma <  \frac{2-b}{N-2}$ ($0 < \sigma < \infty$, if $N=2$) and $0 < b < \frac{N}{3}$, if  $N =2,3$ or $0 < b < 2$, if $N \geq 4$. In addition, in the intercritical case, he also showed a small data global theory in $H^1(\mathbb{R}^N)$ for $N \geq 2$ with the same assumptions on the parameter $b$. For the local theory in $H^1(\Real^N)$, the range of $b$ was extended by \citet{Boa_Dinh}\footnote{Dinh also improved considered the case $N=3$ and $\frac{1}{2}<b<\frac{3}{2}$, however with the extra assumption $\sigma<\frac{3-2b}{2b-1}$.} in dimension $N=2$ to $0<b<1$ and by \citet{cholee2019stability}\footnote{It is worth mentioning that Cho-Lee studied the INLS equation with a potencial $i \partial u_t + \Delta u+Vu + |x|^{-b} |u|^{2 \sigma}u = 0$.} in dimension $N=3$ to $0<b<\frac{3}{2}$. In all these results the range of $b$ is more restricted than the one in \citet{g_8}, where the authors considered $0 < b < 2$. However, the works \cite{Boa}, \cite{Boa_Dinh} and \cite{cholee2019stability} obtained the extra information that the solutions belong to the spaces $\in L^q\left([-T,T];L^r \right)$ for any $L^2$-admissible pair $(q,p)$ satisfying
\begin{equation*}
\frac{2}{q}=\frac{N}{2}-\frac{N}{p},
\end{equation*}
where
\begin{equation}\label{L2adm}
\left\{\begin{array}{cl}
2\leq & p  \leq \frac{2N}{N-2}\hspace{0.5cm}\textnormal{if}\;\;\;  N\geq 3,\\
2 \leq  & p < \infty\;  \hspace{0.5cm}\textnormal{if}\;\; \;N=2,\\
2 \leq & p \leq  \infty\;  \hspace{0.5cm}\textnormal{if}\;\;\;N=1.
\end{array}\right.
\end{equation}

More recently, \citet{campos2019scattering} also proved a small data global theory in $H^1(\mathbb{R}^N)$ in the intercritical regime improving the range of $b$ to $0<b<3/2$ in dimension $N=3$. Inspired by this last result, our first goal of this paper, is to improve the small data global result of \cite{Boa} in the intercritical 2D to the whole range of $b$ where local well-posedness was obtained by \cite{Boa_Dinh}, that is $0<b<1$. More precisely. 
\begin{thm}\label{GWP} Assume $N=2$ and $0<b<1$. Suppose $\frac{2-b}{2}<\sigma<\infty$ and $u_0 \in H^1(\mathbb{R}^N)$ satisfies $\|u_0\|_{H^1}\leq \eta$, for some $\eta>0$. Then there exists $\delta=\delta(\eta)>0$ such that if 
$
\|e^{it\Delta}u_0\|_{S(\dot{H}^{s_c})}<\delta,
$
then there exists a unique global solution $u$ of \eqref{PVI} such that\footnote{Here, as usual, $e^{it\Delta}$ denotes the unitary group associated with the linear Schr\"odinger equation $i \partial_t u +\Delta u = 0$.}
\begin{equation*}
\|u\|_{S(\dot{H}^{s_c})}\leq  2\|e^{it\Delta}u_0\|_{S(\dot{H}^{s_c})}\;\;\;\;\textnormal{and}\qquad
\|u\|_{S\left(L^2\right)}+\|\nabla  u\|_{S\left(L^2\right)}\leq 2c\|u_0\|_{H^1},
\end{equation*}
for some universal constant $c>0$. 
\end{thm}


Other issues, such as, scattering, minimal mass blow-up solutions and concentration were also investigated for the INLS equation. The second and third authors in \cite{farah2017scattering1}-\cite{farah2019scattering} and \citet{campos2019scattering} proved, for different ranges on the parameter $N$, $\sigma$ and $b$, that radial solutions of the IVP \eqref{PVI} scatter in $H^1(\Real^N)$ in the intercritical case. The radial assumption was removed by Miao, Murphy and Zheng \cite{Murphy20} in the 3D cubic setting. \citet{CG16} obtained the classification of minimal mass finite time blow-up solutions for mass-critical INLS equation. Note that \citet{genoud2012critical} proved the existence of minimal mass finite time blow-up solutions based on the pseudo-conformal transformation applied to a standing wave solution. Finally, Campos and the first author \cite{campos2018critical} studied, also in the mass-critical case, the $L^2$-norm concentration of finite time blow-up solutions. 

Another main purpose of this work is to study some dynamical properties of the blow-up solutions to \eqref{PVI} with initial data in $\dot{H}^{s_c}(\Real^N)\cap \dot{H}^1(\Real^N) $, $0<s_c<1$. To this end, we first need a local theory in this space, since, in view of the lack of mass conservation, this is not a trivial consequence of the local theory in $H^1(\Real^N)$.  
We prove the following result.
\begin{thm}\label{LWP} Let $N\geq3$, $0<b<\min\{\frac{N}{2},2\}$ and $\frac{2-b}{N}<\sigma<\frac{2-b}{N-2}$.  If $u_0 \in \dot{H}^{s_c}(\Real^N)\cap \dot{H}^1(\Real^N) $, then there exist $T > 0$ and a unique solution u to \eqref{PVI} satisfying
	$$
	u\in C\left([-T,T];\dot{H}^{s_c}\cap \dot{H}^1 \right)\bigcap L^q\left([-T,T];\dot{H}^{s_c,p}\cap \dot{H}^{1,p} \right)\bigcap L^a\left([-T,T]; L^r \right),$$
	for any ($q,p$) $L^2$-admissible and $(a,r)$ $\dot H^{s_c}$-admissible\footnote{See \eqref{CPA1}-\eqref{CPA2} below for the definition of $\dot H^{s_c}$-admissible pair.}. 
\end{thm}

To prove Theorems \ref{GWP}-\ref{LWP}, we use the contraction mapping argument based on the Strichartz estimates related to the linear problem. Here and in what follows, by a solution of the IVP \eqref{PVI} with $u_0\in X$ ($X=H^1(\Real^N) \mbox{ or } X=\dot H^{s_c}(\Real^N)\cap \dot H^1(\Real^N)$), we mean a function $u\in C(I; X)$ on some interval $I\ni 0$ that satisfies the Duhamel formula given by
\begin{align}\label{duhamel}
u(t)=e^{it\Delta}u_0+i\int_{0}^{t}e^{i(t-t')\Delta}|x|^{-b}|u|^{2\sigma}u(t')\,dt', \quad \textnormal{for}\quad  t\in I.
\end{align}

It is worth noticing that the local theory stated in Theorem \ref{LWP} also holds for the defocusing inhomogeneous nonlinear Schr\"odinger (INLS) equation $i \partial_t u +\Delta u - |x|^{-b} |u|^{2\sigma}u = 0$. Moreover, it is still an open problem to obtain the same result for $N=1,2$ in both focusing and defocusing cases.

After we established the local theory in $\dot H^{s_c}(\Real^N)\cap \dot H^{1}(\Real^N)$ for the intercritical INLS equation, we study the asymptotic behavior of the solutions. We recall that Campos and the first author in \cite{campos2018critical} established, using a Sobolev embedding (see Stein-Weiss \citep[Theorem B*]{stein}), the following inequality
\begin{equation}\label{GNcritintr}
\int |x|^{-b}|f|^{2\sigma+2} \, dx\leq c \|\nabla f\|^2_{L^2}\|f\|^{2\sigma}_{L^{\sigma_c}},
\end{equation}
for some $c>0$ and all functions $f\in \dot H^1(\mathbb{R}^N)\cap L^{\sigma_c}(\mathbb{R}^N)$. First, we obtain the best constant for the above inequality.  More precisely, we have the following sharp Gagliardo-Nirenberg type estimate. 
\begin{thm}\label{GNU}
Let $N\geq 1$, $0<b<2$, $\frac{2-b}{N}<\sigma<\frac{2-b}{N-2}$ ($\frac{2-b}{N}<\sigma<\infty$, if $N=1,2$) and $\sigma_c=\frac{2N\sigma}{2-b}$, then the following Gagliardo-Nirenberg inequality holds for all $f\in \dot H^1(\Real^N)\cap L^{\sigma_c}(\Real^N)$
\begin{align}\label{GNsc}
\int_{\Real^N} |x|^{-b} |f(x)|^{2\sigma+2}\,dx\leq \frac{\sigma+1}{\|V\|_{L^{\sigma_c}}^{2\sigma}}\|\nabla f\|_{L^2}^2\|f\|_{L^{\sigma_c}}^{2\sigma},
\end{align}
 where $V$ is a solution to the elliptic equation, 
\begin{align}\label{elptcpc}
\Delta V+|x|^{-b}|V|^{2\sigma}V-|V|^{\sigma_c-2}V=0
\end{align}
with minimal $L^{\sigma_c}$-norm.


\end{thm}

Although uniqueness of solutions for the elliptic equation \eqref{elptcpc} is not known,  this will not be an issue to our purpose since the sharp constant depends only on the $L^{\sigma_c}$-norm of the solution. The proof of Theorem \ref{GNU} relies mainly in a compact embedding result (see Proposition \ref{WSC} below) generalizing the one obtained by \citet{g_8}. Moreover, since $\dot H^{s_c}(\mathbb{R}^N)\subset  L^{\sigma_c} (\mathbb{R}^N)$, Theorem \ref{GNU} allows us to establish sufficient conditions
for global existence in $\dot H^{s_c}(\mathbb{R}^N)\cap \dot H^1(\mathbb{R}^N)$. 
\begin{thm}\label{global}
Let $N\geq3$, $0<b<\min\{\frac{N}{2},2\}$, $\frac{2-b}{N}<\sigma<\frac{2-b}{N-2}$, $s_c=\frac{N}{2}-\frac{2-b}{2\sigma}$ and $\sigma_c=\frac{2N\sigma}{2-b}$. For $u_0\in \dot H^{s_c}(\mathbb{R}^N)\cap \dot H^1(\mathbb{R}^N)$, let $u(t)$ be the corresponding solution to \eqref{PVI} given by Theorem \ref{LWP} and $T^*>0$ the maximal time of existence. Suppose that $\sup_{t\in [0,T^*)}\|u(t)\|_{\dot H^{s_c}}<\|V\|_{L^{\sigma_c}}$, where $V$ is a solution of the elliptic equation \eqref{elptcpc} with minimal $L^{\sigma_c}$-norm. Then $u(t)$ exists globally in the time.
\end{thm}

Finally, we treat the phenomenon of $L^{\sigma_c}$-norm concentration in the intercritical regime for finite time blow-up solutions in $\dot H^{s_c}(\mathbb{R}^N)\cap \dot H^1(\mathbb{R}^N)$. We first recall that in $H^1(\Real^N)$ a simple criterion for the existence of finite time blow-up solutions to \eqref{PVI} was obtained by the second author in \cite{Farah_well}. Indeed, considering $u_0\in \Sigma:=\{f\in H^1;\,\,|x|f\in L^2 \}$, then the corresponding solution to \eqref{PVI} satisfies the virial identity
\begin{align}
\frac{d^2}{dt}\int |x|^2|u(x,t)|^2=8(2\sigma s_c+2)E[u_0]-8\sigma s_c\|\nabla u(t)\|_{L^2}^2.
\end{align}
From this identity, we immediately see that, in the intercritical case, if $E\left[u_0 \right] < 0$, then the graph of $t \mapsto \int |x|^2 |u|^2$ lies below a parabola whose concavity is facing
down, which becomes negative in finite time. Therefore, the solution cannot exist globally and blows up in finite time. \citet{D18} proved the same result assuming radial negative energy initial data (also, when $N=1$, radial symmetry can also be removed). Since $H^1(\Real^N) \subset \dot H^{s_c}(\mathbb{R}^N)\cap \dot H^1(\mathbb{R}^N)$, these results also ensure the existence of finite time blow-up solutions to \eqref{PVI} for some initial data $u_0\in \dot H^{s_c}(\mathbb{R}^N)\cap \dot H^1(\mathbb{R}^N)$.

Here, as in \citet{guo2013note}, we suppose that the finite time blow-up solution to \eqref{PVI} is of type II, that is, the maximal time of existence $T^{\ast}>0$ is finite and the critical norm remains bounded
\begin{align}\label{condbounded}
\sup_{t\in [0,T^*)}\|u(t)\|_{\dot H^{s_c}}<\infty.
\end{align} 
It must be noted that \citet{MR_Bsc} showed the existence of radially symmetric finite time blow-up solutions to the NLS equation in $\dot H^{s_c}(\mathbb{R}^N)\cap \dot H^1(\mathbb{R}^N)$ that \eqref{condbounded} does not occurs. Furthermore, they provided a lower bound for blow-up rate of the critical norm for these solution. However, their proof does not apply to the non-radial case and it may be possible to have finite time blow-up solutions satisfying \eqref{condbounded} in this case.

As a consequence of the proof of Theorem \ref{global} (see Remark \ref{RemGWP}), if there exists a solution $u(t)$ that blows up in finite time $T^*>0$ satisfying \eqref{condbounded}, then we must have
\begin{align}
\sup_{t\in [0,T^*)}\|u(t)\|_{L^{\sigma_c}}\geq \|V\|_{L^{\sigma_c}},
\end{align}
where $V$ is a solution to elliptic equation \eqref{elptcpc} with minimal $L^{\sigma_c}$ - norm.
This suggests us to investigate the occurrence of the $L^{\sigma_c}$-norm concentration for finite time blow-up solutions satisfying \eqref{condbounded}. \citet{guo2013note} obtained such concentration for the NLS equation in the intercritical case (and without radial symmetry), partially generalizing the results obtained by \citet{holmer2007blow}, which deals with the radial 3D cubic NLS equation (see also Campos the first author \cite{campos2018critical} for a similar result in the INLS setting).


Using a profile decomposition technique several authors has studied the critical norm concentration of finite time blow-up solutions (see e.g \citet{HM}, \citet{guo2013note}, Pigott and the second author \cite{farah2016nonlinear}, Campos and the first author \cite{campos2018critical} and  \citet{dinh2018study}) for various dispersive models. However, in this work we use a different approach based on a compact embedding result (see Proposition \ref{WSC} below). Our main result in this direction is the following.

\begin{thm}\label{concentration} Let $N\geq3$, $0<b<\min\{\frac{N}{2},2\}$ and $\frac{2-b}{N}<\sigma<\frac{2-b}{N-2}$. For $u_0\in \dot H^{s_c}(\mathbb{R}^N)\cap \dot H^1(\mathbb{R}^N)$,  let $u(t)$ be the corresponding solution to \eqref{PVI} given by Theorem \ref{LWP} and assume that it blows up in finite time $T^*>0$ satisfying
\eqref{condbounded}.
If $\lambda (t)>0$ is such that
\begin{align}\label{lamb}
\lambda(t)\|\nabla u(t)\|_{L^2}^{\frac{1}{1-s_c}}\to  \infty, \quad \textit{as} \quad t\to T^*,
\end{align}
then,
\begin{align}\label{Conc0}
\liminf_{t\to T^*} \int_{|x|\leq \lambda(t)}|u(x,t)|^{\sigma_c}\,dx\geq \|V\|^{\sigma_c}_{L^{\sigma_c}},
\end{align}
where $V$ is a minimal $L^{\sigma_c}$-norm solution to the elliptic equation \eqref{elptcpc}.
\end{thm}
\begin{rem} Two important comments about this result have to be emphasized. First, as a consequence of Theorem \ref{LWP} and under the assumption \eqref{condbounded}, it is possible to deduce the existence of $\lambda(t)>0$ such that $\lambda(t)\rightarrow 0$, as $t\rightarrow T^*$. Indeed, $\lambda(t)=(T^*-t)^{\alpha}$ with $0<\alpha<\frac{1}{2}$ satisfies the assumptions of Theorem \ref{concentration} (see Remark \ref{eqtime} below). Second, the inequality \eqref{Conc0} asserts that the concentration occurs at the origin even for non radial finite time blow-up solutions satisfying \eqref{condbounded}. This is due to our method of proof based on the compact embedding result stated in Proposition \ref{WSC}. Recall that for the NLS equation concentration at the origin appears in the radial case (see \citet{MT} and \citet{T90}), however, for the non radial case \citet{HM} and \citet{guo2013note} obtained that the concentration occurs at some point of space (not necessary at the origin).
\end{rem}

We want to point out that the last two theorems are stated under the same assumptions as in the local theory in $\dot H^{s_c}(\mathbb{R}^N)\cap \dot H^1(\mathbb{R}^N)$ from Theorem \ref{LWP}. However, if one can improve the range of the parameters in Theorem \ref{LWP}, then Theorems \ref{global}-\ref{concentration} will be also true (with the same proof given here) in the same range. This is due to the fact that the sharp Gagliardo-Nirenberg type inequality given by Theorem \ref{GNU} (and also the compact embedding result by Proposition \ref{WSC} below) holds for the intercritical case in all dimensions $N\geq 1$ and $0<b<2$.

The rest of the paper is organized as follows. In Section \ref{sec2}, we introduce some notations and preliminary estimates. In Section \ref{sec3}, we obtain the well-posedness results stated in Theorems \ref{GWP}-\ref{LWP}. In Section \ref{sec4}, we prove the sharp Gagliardo-Nirenberg inequality in Theorem \ref{GNU} and use it to deduce Theorem \ref{global}. Section \ref{CNC} is devoted to the proof of Theorem \ref{concentration}. Finally, in Section \ref{ACR} we present another concentration result for special solutions of the INLS equation.

\section{Notation and Preliminaries}\label{sec2}
We start this section by introducing the notation used throughout the
paper.  We use $c$ to denote various constants that may vary line by line. Let $a$ and $b$ be positive real numbers, the
notation $a \lesssim b$ means that there exists a positive constant $c$ such that $a \leq cb$. Given a real number $r$, we use $r^+$ and $r^-$ to denote $r+\varepsilon$ and $r-\varepsilon$, respectively, for some $\varepsilon>0$ sufficiently small. For a subset $A\subset \mathbb{R}^N$, its complement is denoted by $A^C=\mathbb{R}^N \backslash A$ and the characteristic function $\chi_A(x)$ denotes the function that has value 1 at points of A and 0 at points of $A^C$. Given $x,y \in \mathbb{R}^N$, $x \cdot y$ denotes the usual inner product of $x$ and $y$ in $\mathbb{R}^N$. For a number $p\in [1,\infty]$ we denote its H\"older dual by $p'=\frac{p}{p-1}$, satisfying $\frac{1}{p}+\frac{1}{p'}=1$.

We use $\|f\|_{L^p}$ to denote the $L^p(\mathbb{R}^N)$ norm. The Schwartz class functions is denoted by $\mathcal S(\Real^N)$. The norm in the Sobolev spaces $H^{s,p}=H^{s,p}(\mathbb{R}^N)$ and $\dot{H}^{s,p}=\dot{H}^{s,p}(\mathbb{R}^N)$, are defined, respectively, by $\|f\|_{H^{s,p}}:=\|J^sf\|_{L^p}$ and $\|f\|_{\dot{H}^{s,p}}:=\|D^sf\|_{L^p},$
where  $J^s$ and $D^s$ stand for the Bessel and Riesz potentials of order $s$, given via Fourier transform by $\widehat{J^s f}=(1+|\xi|^2)^{\frac{s}{2}}\widehat{f}$ and $\widehat{D^sf}=|\xi|^s\widehat{f}.
$
If $p=2$ we denote $H^{s,2}$ and $\dot{H}^{s,2}$ simply by $H^s$ and  $\dot{H}^{s}$, respectively.

Let $ q,p >0$, $s\in \mathbb{R}$, and $I\subset \mathbb{R}$ an interval; the mixed norms in the spaces $L^q_{I}L^p_x$ and $L^q_{I} H^s_x$ of a function $f=f(x,t)$ are defined as
$$
\|f\|_{L^q_{I}L^p_x}=\left(\int_I\|f(\cdot, t)\|^q_{L^p_x}dt\right)^{\frac{1}{q}}
\qquad
\mbox{and}
\qquad
\|f\|_{L^q_{I}H^s_x}=\left(\int_I\|f(\cdot, t)\|^q_{H^s_x}dt\right)^{\frac{1}{q}},
$$
with the usual modifications if either $q=\infty$ or $p=\infty$. When the  $x$-integration is restricted to a subset $A\subset\mathbb{R}^N$ then the Lebesgue norm and the mixed norm will be denoted by $\|f\|_{L^p(A)}$ and $\|f\|_{L_I^qL^r_x(A)}$, respectively. Moreover, if $I=\mathbb{R}$ we shall use the notations $\|f\|_{L_t^qL^p_x}$ and $\|f\|_{L_t^qH^s_x}$.

We now recall some useful inequalities.
\begin{lemma}\textbf{(Sobolev embedding)}\label{SI} Let $s>0$ and $1\leq p<\infty$.
\begin{itemize}
\item [(i)] If $s\in (0,\frac{N}{p})$ then $H^{s,p}(\mathbb{R}^N)$ is continuously embedded in $L^r(\mathbb{R^N})$ where $s=\frac{N}{p}-\frac{N}{r}$. Moreover, 
\begin{equation}\label{SEI} 
\|f\|_{L^r}\leq C(N,s)\|D^sf\|_{L^{p}}.
\end{equation}
\item [(ii)] If $s=\frac{N}{2}$ then $H^{s}(\mathbb{R}^N)\subset L^r(\mathbb{R^N})$ for all $r\in[2,\infty)$. Furthermore,
\begin{equation}\label{SEI1} 
\|f\|_{L^r}\leq c\|f\|_{H^{s}}.
\end{equation}
\end{itemize}
\begin{proof} We refer to \citet[Theorem $6.5.1$]{BERLOF} for a complete proof (see also \citet[Theorem $3.3$]{LiPo15} and \citet[Proposition 4.18]{DEMENGEL}). 
\end{proof}
\end{lemma}
In particular, we have 
\begin{align}\label{SEsc}
\|f\|_{L^p}\leq C(N,s)\|f\|_{\dot H^s},\quad \textnormal{for all} \quad f\in \dot H^s(\Real^N),
\end{align}
where $p=\frac{2N}{N-2s}$. Moreover, for $s_c=\frac{N}{2}-\frac{2-b}{2\sigma}$, we have $\dot H^{s_c}(\mathbb{R}^N)\subset L^{\sigma_c}(\mathbb{R}^N)$, where $\sigma_c=\frac{2N\sigma}{2-b}=\frac{2N}{N-2s_c}.$

\begin{lemma}\textbf{(Gagliardo-Nirenberg's inequality)}\label{GNgeral}
	 Consider $1\leq p,q,r\leq \infty$ and let $j,m$ be two integers, $0\leq j<m$. If 
	\begin{align}
	\frac{1}{q}-\frac{j}{N}=\theta\left(\frac{1}{r}-\frac{m}{N}\right)+\frac{1-\theta}{p},
	\end{align}
	for some $\theta\in \left[\frac{j}{m},1\right]$ ($\theta<1$ if $r>1$ and $m-j-\frac{N}{r}=0$), then there exists a constant $c=c(j,m,p,q,r)$ such that
	\begin{align}\label{GNgeral2}
	\sum_{|\alpha|=j}\|D^\alpha f\|_{L^q}\leq c\left(\sum_{|\beta|=m}\|D^{\beta} f\|_{L^r}\right)^\theta\|f\|_{L^p}^{1-\theta}
	\end{align}
	for all $f\in \mathcal S(\Real^N)$.
\end{lemma}
\begin{proof}
	See \citet[Theorem 1.3.7]{cazenave} and \citet{Nirenberg} .
\end{proof}

Next, we recall some Strichartz type estimates associated to the linear Schr\"odinger propagator (see also \citet{HRasharp}, \citet{GUEVARA} and \cite{farah2019scattering}). Given $s> 0$, we say that a pair $(q,p)$ is $\dot{H}^s$-admissible if
\begin{equation}\label{CPA1}
\frac{2}{q}=\frac{N}{2}-\frac{N}{p}-s,
\end{equation}
where 
\begin{equation}\label{CPA2}
\left\{\begin{array}{cl}
\frac{2N}{N-2s}\leq & p  \leq \left(\frac{2N}{N-2}\right)^-,\;\hspace{0.4cm}\textnormal{if}\;\;\;  N\geq 3,\\
\frac{2}{1-s} \leq  & p \leq \left((\frac{2}{1-s})^+\right)',\;  \hspace{0.2cm}\textnormal{if}\;\; \;N=2,\\
\frac{2}{1-2s} \leq & p \leq  \infty,\; \; \hspace{1.2cm}\textnormal{if}\;\;\;N=1.
\end{array}\right.
\end{equation}
In the same way, we say that $(q,p)$ is $\dot{H}^{-s}$-admissible if 
$$
\frac{2}{q}=\frac{N}{2}-\frac{N}{p}+s,
$$
where
\begin{equation}\label{H-s}
\left\{\begin{array}{cl}
\left(\frac{2N}{N-2s}\right)^{+}\leq & p  \leq \left(\frac{2N}{N-2}\right)^-,\;\;\hspace{0.4cm}\textnormal{if}\;\;  N\geq 3,\\
\left(\frac{2}{1-s}\right)^{+} \leq  & p \leq \left((\frac{2}{1+s})^+\right)',\;  \hspace{0.2cm}\textnormal{if}\;\; \;N=2,\\
\left(\frac{2}{1-2s}\right)^{+} \leq & p \leq  \infty,\; \; \hspace{1.2cm}\textnormal{if}\;\;\;N=1.
\end{array}\right.
\end{equation}
Now for $s\in \mathbb{R}$, let $\mathcal{A}_s=\{(q,p);\; (q,p)\; \textnormal{is} \;\dot{H}^s-\textnormal{admissible}\}$\footnote{The restriction for $(q,p)$ $\dot{H}^0$-{admissible} is given by \eqref{L2adm}.}. We define the spaces $S(\dot{H}^{s})$ and $S'(\dot{H}^{-s})$ equipped with the following Strichartz norm
\begin{equation}\label{StrNorm}
\|u\|_{S(\dot{H}^{s})}=\sup_{(q,p)\in \mathcal{A}_{s}}\|u\|_{L^q_tL^p_x} 
\end{equation}
and the dual Strichartz norm
$$
\|u\|_{S'(\dot{H}^{-s})}=\inf_{(q,p)\in \mathcal{A}_{-s}}\|u\|_{L^{q'}_tL^{p'}_x},
$$
where $(q',p')$ is the H\"older dual to $(q,p)$. We denote $S(\dot{H}^{0})$ by $S(L^2)$. To indicate a restriction to a time interval $I\subset \mathbb{R}$, we will write $S(\dot{H}^{s};I)$ and $S'(\dot{H}^{-s};I)$.

One of the main tools we use in the proof of our local and global well-posedness theory are the well-known Strichartz estimates. 
\begin{lemma}\label{Lemma-Str} The following statements hold.
 \begin{itemize}
\item [(i)] (Linear estimates).
\begin{equation}\label{SE1}
\| e^{it\Delta}f \|_{S(L^2)} \leq c\|f\|_{L^2},
\end{equation}
\begin{equation}\label{SE2}
\|  e^{it\Delta}f \|_{S(\dot{H}^s)} \leq c\|f\|_{\dot{H}^s}.
\end{equation}
\item[(ii)] (Inhomogeneous estimates).
\begin{equation}\label{SE3}					 
\left \| \int_{\mathbb{R}} e^{i(t-t')\Delta}g(\cdot,t') dt' \right\|_{S(L^2)}\;+\; \left \| \int_{0}^t e^{i(t-t')\Delta}g(\cdot,t') dt' \right \|_{S(L^2) } \leq c\|g\|_{S'(L^2)},
\end{equation}
\begin{equation}\label{SE5}
\left \| \int_{0}^t e^{i(t-t')\Delta}g(\cdot,t') dt' \right \|_{S(\dot{H}^s) } \leq c\|g\|_{S'(\dot{H}^{-s})}.
\end{equation}
\end{itemize}
\end{lemma}

For a complete proof we refer the reader to \citet{LiPo15} and \citet{KATO2} (see also \citet{HRasharp}, \cite{Boa} and the references therein).

%

\section{Well-posedness theory}\label{sec3}
In this section we prove the well-posedness results stated in Theorems \ref{GWP}-\ref{LWP}. The proofs follow from a contraction mapping argument based on the Strichartz estimates. In view of the singular factor $|x|^{-b}$ in the nonlinearity, we frequently divide our analysis in two regions. Indeed, let $B=B(0,1)=\{ x\in \mathbb{R}^N;|x|\leq 1\}$ a simple computation revels that 
\begin{equation}\label{RIxb}
\left\||x|^{-b}\right\|_{L^\gamma(B)}<\infty,\;\;\;\textnormal{if}\;\;\frac{N}{\gamma}-b>0 \quad \textnormal{and} \quad \left\||x|^{-b}\right\|_{L^\gamma(B^C)}<\infty,\;\;\;\textnormal{if}\;\;\frac{N}{\gamma}-b<0.
\end{equation}
We are going to use these facts several times throughout this section.
\subsection{Global well-posedness in $H^1(\Real^N)$}
In this subsection, we turn our attention to proof the Theorem \ref{GWP}. The heart of the proof is to establish good estimates on the nonlinearity $|x|^{-b}|u|^{2\sigma}u$. 
The next lemma provides these estimates. 
 \begin{lemma}\label{LG1} 
Let $N=2$, $0<b<1$ and $\frac{2-b}{2}<\sigma<\infty$. Then there exist $c>0$ and $\theta\in (0,2\sigma)$ sufficiently small such that 
\begin{itemize}
\item [(i)] $\left \|\chi_B|x|^{-b}|u|^{2\sigma} v \right\|_{S'(\dot{H}^{-s_c})}+\left \|\chi_{B^C}|x|^{-b}|u|^{2\sigma} v \right\|_{S'(\dot{H}^{-s_c})} \leq c\| u\|^{\theta}_{L^\infty_tH^1_x}\|u\|^{2\sigma-\theta}_{S(\dot{H}^{s_c})} \|v\|_{S(\dot{H}^{s_c})}$,
\item [(ii)] $\left\|\chi_B|x|^{-b}|u|^{2\sigma}v \right\|_{S'(L^2)}+\left\|\chi_{B^C}|x|^{-b}|u|^{2\sigma} v \right\|_{S'(L^2)}\leq c\| u\|^{\theta}_{L^\infty_tH^1_x}\|u\|^{2\sigma-\theta}_{S(\dot{H}^{s_c})} \| v\|_{S(L^2)}$,
\item [(iii)] $\left\|\nabla (|x|^{-b}|u|^{2\sigma} u)\right\|_{L_t^{q'}L^{r'}_x}\leq c\| u\|^{\theta}_{L^\infty_tH^1_x}\|u\|^{2\sigma-\theta}_{S(\dot{H}^{s_c})} \|\nabla u\|_{S(L^2)}+c\| u\|^{\theta+1}_{L^\infty_tH^1_x}\|u\|^{2\sigma-\theta}_{S(\dot{H}^{s_c})}$,
\end{itemize}
where\footnote{Note that the pair $(q,r)$ is $L^2$-admissible.} $(q,r)=\left(\frac{2}{1-\theta},\frac{2}{\theta}\right)$.
\end{lemma} 
\begin{proof} To prove $(i)$ and $(ii)$, we first define the following numbers
\begin{equation}\label{LGWP1}
\widehat{q}=\frac{4\sigma(2\sigma+2-\theta)}{2\sigma(2\sigma+b)-\theta(2\sigma-2+b)},\;\qquad\; \widehat{r}=\frac{4\sigma(2\sigma+2-\theta)}{(2\sigma-\theta)(2-b)}
\end{equation}
and
\begin{equation}\label{LGWP2}
\widehat{a}=\frac{2\sigma(2\sigma+2-\theta)}{2-b},\;\qquad\;\widetilde{a}=\frac{2\sigma(2\sigma+2-\theta)}{2\sigma(2\sigma+b-\theta)-(2-b)(1-\theta)}.
\end{equation}
It is easy to see that, for $\theta$ sufficiently small, $(\widehat{q},\widehat{r})$ is $L^2$-admissible, $(\widehat{a},\widehat{r})$ is $\dot{H}^{s_c}$-admissible and $(\widetilde{a},\widehat{r})$ is $\dot{H}^{-s_c}$ admissible. Moreover
\begin{equation}\label{GWP3}
\frac{1}{\widetilde{a}'}=\frac{2\sigma-\theta}{\widehat{a}}+\frac{1}{\widehat{a}}\;\quad \textnormal{and}\;\quad \frac{1}{\widehat{q}'}=\frac{2\sigma-\theta}{\widehat{a}}+\frac{1}{\widehat{q}}.    
\end{equation}

Let us prove $(i)$. Let $A\subset \mathbb{R}^N$ denotes either $B$ or $B^C$. By definition of $S'(\dot{H}^{-s})$, we clearly have
 $
 \left \|\chi_A|x|^{-b}|u|^{2\sigma} v \right\|_{S'(\dot{H}^{-s_c})}\leq \left \|\chi_A|x|^{-b}|u|^{2\sigma} v \right\|_{L^{\widetilde{a}'}_tL^{\widehat{r}'}_x}.
 $
 On the other hand, from  H\"older's inequality we deduce
 \begin{equation}\label{LG1Hs1}
\begin{split}
\left \| \chi_A |x|^{-b}|u|^{2\sigma} v \right\|_{L^{\widehat{r}'}_x} &\leq \left\||x|^{-b}\right\|_{L^\gamma(A)}  \|u\|^{\theta}_{L^{\theta r_1}_x}   \|u\|^{2\sigma-\theta}_{L_x^{(2\sigma-\theta)r_2}}  \|v\|_{L^{\widehat{r}}_x}  \\
&=\left\||x|^{-b}\right\|_{L^\gamma(A)}  \|u\|^{\theta}_{L^{\theta r_1}_x}   \|u\|^{2\sigma-\theta}_{L_x^{\widehat{r}}} \|v\|_{L^{\widehat{r}}_x},
\end{split}
\end{equation}
where
\begin{equation}\label{LG1Hs2}
\frac{1}{\widehat{r}'}=\frac{1}{\gamma}+\frac{1}{r_1}+\frac{1}{r_2}+\frac{1}{\widehat{r}}\;\;\textnormal{and}\;\;\widehat{r}=(2\sigma-\theta)r_2.
\end{equation}
Observe that \eqref{LG1Hs2} implies
\begin{equation*}
	\frac{2}{\gamma}=2-\frac{2(2\sigma+2-\theta)}{\widehat{r}}-\frac{2}{r_1},
\end{equation*}
and using the value of $\widehat{r}$, it follows that
\begin{equation}\label{LG1Hs3}
\frac{2}{\gamma}-b=\frac{\theta(2-b)}{2\sigma}-\frac{2}{r_1}.
\end{equation}
Next we show that $\left\||x|^{-b}\right\|_{L^\gamma(A)}$ is finite and  $H^1\subset L^{\theta r_1}$. Since $2\sigma>2-b$ we have $\frac{4\sigma}{2-b}>2$. Thus, if $A=B$ and choosing $\theta r_1\in \left(\frac{4\sigma}{2-b},\infty\right)$, from \eqref{LG1Hs3}, we immediately get $\frac{2}{\gamma}-b>0$. Furthermore, if $A=B^C$ and choosing $\theta r_1\in\left(2,\frac{4\sigma}{2-b}\right)$, we obtain $\frac{2}{\gamma}-b<0$. In both cases we have $\left\||x|^{-b}\right\|_{L^\gamma(A)}<\infty$ and, from Lemma \ref{SI}, $H^1\subset L^{\theta r_1}$ (recall that, for $N=2$, one has $H^1\subset L^p$, $p\in [2,\infty)$). Therefore, the inequality \eqref{LG1Hs1} yields
\begin{equation}\label{LG1Hs41}
\left \| \chi_A |x|^{-b}|u|^{2\sigma} v \right\|_{L^{\widehat{r}'}_x} \lesssim \|u\|^{\theta}_{H^1_x}  \|u\|^{2\sigma-\theta}_{L_x^{\widehat{r}}} \|v\|_{L^{\widehat{r}}_x}.
\end{equation}
Now applying H\"older's inequality in time and recalling \eqref{GWP3}, we have
\begin{equation}
\begin{split}\label{LG1Hs42}
\left \|\chi_A  |x|^{-b}|u|^{2\sigma} v \right\|_{L_t^{\widetilde{a}'}L^{\widehat{r}'}_x}
&\lesssim \|u\|^{\theta}_{L^\infty_tH^1_x} \|u\|^{2\sigma-\theta}_{ L_t^{\widehat{a}} L_x^{\widehat{r}}} \|v\|_{L^{\widehat{a}}_tL^{\widehat{r}}_x}   \\
&\lesssim \| u\|^{\theta}_{L^\infty_tH^1_x}\|u\|^{2\sigma-\theta}_{S(\dot{H}^{s_c})} \|v\|_{S(\dot{H}^{s_c})},
\end{split}
\end{equation}
which implies $(i)$.

Since $(\widehat{q},\widehat{r})$ is $L^2$-admissible, the proof of $(ii)$ is essentially the same as $(i)$. It is worth noting that, once \eqref{LG1Hs41} is achieved , we use \eqref{GWP3} to deduce
 \begin{equation}\label{LGHsii}
\left\| \chi_A |x|^{-b}|u|^{2\sigma} v\right \|_{L_t^{\widehat{q}'}L^{\widehat{r}'}_x}\lesssim \|u\|^{\theta}_{L^\infty_tH^1_x}\|u\|^{2\sigma-\theta}_{L_t^{\widehat{a}} L_x^{\widehat{r}}} \|v\|_{L^{\widehat{q}}_tL^{\widehat{r}}_x},
\end{equation}
which yields $(ii$).
 
Before starting the proof of $(iii)$, we need the following numbers
\begin{equation}\label{LGWP3}
\bar{a}=\frac{2(2\sigma+1-\theta)}{1-s_c+\theta},\;\qquad\;\bar{r}=\frac{4\sigma(2\sigma+1-\theta)}{2\sigma(1-b+s_c)+2-b-\theta(2-b+2\sigma)}
\end{equation}
and
\begin{equation}\label{LGWP4}
\bar{q}=\frac{2(2\sigma+1-\theta)}{1+2\sigma s_c+\theta(1-s_c)},\quad a^*=\frac{2(2\sigma-\theta)}{1+\theta},\quad\;r^*=\frac{4\sigma(2\sigma-\theta)}{2\sigma(1-b)-\theta(2-b+2\sigma)}.
\end{equation}
Note that, for $\theta$ sufficiently small, $(\bar{q},\bar{r})$ is $L^2$-admissible, $(\bar{a},\bar{r})$, $(a^*,r^*)$ is $\dot{H}^{s_c}$-admissible and\footnote{Since $b<1$, we have that the denominator of $r^*$ is positive and $\bar{r},r^*>\frac{2}{1-s_c}$ (this is a necessary  condition for  $\dot{H}^{s_c}$-admissible pairs, see \eqref{CPA2}).},
\begin{equation}\label{GWP4}
\frac{1}{q'}=\frac{2\sigma-\theta}{\bar{a}}+\frac{1}{\bar{q}}\;\;\;\;\textnormal{and}\;\;\;\;(2\sigma-\theta)q'=a^*.    
\end{equation}
Again, let $A\subset \mathbb{R}^N$ denotes either $B$ or $B^C$. From \eqref{GWP4}) and H\"older's inequality we deduce
\[
\begin{split}
\left\| \nabla \left( |x|^{-b}|u|^{2\sigma} u\right) \right\|_{L^{q'}_tL^{r'}_x(A)} \leq& \left \|\left\||x|^{-b}\right\|_{L^\gamma(A)}\|u\|_{L_{x}^{r_1\theta}}^{\theta}\|u\|^{2\sigma-\theta}_{L_x^{\bar{r}}} \|\nabla u \|_{L^{\bar{r}}_x}\right\|_{L^{q'}_t}\\
&+ \left\|\left\||x|^{-b-1}\right\|_{L^d(A)}\|u\|_{L_{x}^{(\theta+1)p_1}}^{\theta+1} \|u\|^{2\sigma-\theta}_{L_{x}^{r^*}}\right\|_{L^{q'}_t}\\
 \lesssim & \|u\|_{L_{x}^{r_1\theta}}^{\theta}\|u\|^{2\sigma-\theta}_{L_t^{\bar{a}}L_x^{\bar{r}}} \|\nabla u \|_{L^{\bar{q}}_tL^{\bar{r}}_x}+c\|u\|_{L_{x}^{(\theta+1)p_1}}^{\theta+1} \|u\|^{2\sigma-\theta}_{L_{t}^{a^*}L_{x}^{r^*}},
\end{split}
\]
where
$$
\frac{1}{r'}=\frac{1}{\gamma}+\frac{1}{r_1}+\frac{2\sigma-\theta}{\bar{r}}+\frac{1}{\bar{r}}=\frac{1}{d}+\frac{1}{p_1}+\frac{2\sigma-\theta}{r^*}.
$$
Using the definition of the numbers $\bar{r}$ and $r^*$ one has
$$
\frac{2}{\gamma}-b=\frac{\theta(2-b)}{2\sigma}-\frac{2}{r_1}, \qquad \frac{2}{d}-b-1=\frac{\theta(2-b)}{2\sigma}-\frac{2}{p_1},
$$
which are analogous to the relation \eqref{LG1Hs3}. Finally, choosing $r_1$ and $p_1$ as in $(i)$ we have that $\left\||x|^{-b}\right\|_{L^\gamma(A)}$ and $\left\||x|^{-b-1}\right\|_{L^d(A)}$ are finite. Also, $H^1\subset L^{\theta r_1} \cup L^{(\theta+1)p_1}$ and we complete the proof of $(iii)$.
\end{proof}

It should be emphasized that the third author in \cite{Boa} proved the previous lemma under the assumption $0<b<\frac{2}{3}$. Here we extend it to $0 < b <1$.

Now, we have all the tools to prove Theorem \ref{GWP}.

\begin{proof}[Proof of Theorem \ref{GWP}]
First note that $|x|^{-b}|u|^{2\sigma}u=\chi_{B} |x|^{-b}|u|^{2\sigma}u+\chi_{B^c} |x|^{-b}|u|^{2\sigma}u$. So, applying Lemma \ref{Lemma-Str}, we have for $I=[-T,T]$
\begin{align*}
\left\| \int_{0}^{t}e^{i(t-t')\Delta} |x|^{-b}|u|^{2\sigma}u(t')\,dt' \right\|_{S\left(L^2;I\right)}
&\lesssim \left\|\chi_{B} |x|^{-b}|u|^{2\sigma}u\right\|_{S'\left(L^2;I \right)}+\left\|\chi_{B^c} |x|^{-b}|u|^{2\sigma}u\right\|_{S'\left(L^2;I\right)}.
\end{align*}
Similarly, we can use the same argument to estimate the norm $\|\cdot\|_{S(\dot{H}^{s_c};I)}$. The rest of the proof follows the same lines as in \citet[Corollary 1.12]{Boa}.
\end{proof}
\subsection{Local well-posedness in $\dot{H}^{s_c}(\Real^N)\cap \dot{H}^1(\Real^N)$}
Recall that $s_c=\frac{N}{2}-\frac{2-b}{2\sigma}$. In this subsection, we show that the IVP \eqref{PVI} is locally well-posed in $\dot{H}^{s_c}(\Real^N)\cap \dot{H}^1(\Real^N)$, for $N\geq 3$ and $\frac{2-b}{N}<\sigma<\frac{2-b}{N-2}$ (equivalently $0<s_c<1$). We start with some estimates for the gradient of the nonlinearity.


\begin{lemma}\label{lema1} Let $N\geq 3 $, $0<b<\min\{\frac{N}{2},2\}$ and $\frac{2-b}{N}<\sigma<\frac{2-b}{N-2}$, then there exist $c,\theta_1,\theta_2 >0$ such that the following inequality holds  
	\begin{align}\label{H1Sl2}
	\left \|\chi_{B^c} \nabla |x|^{-b}|u|^{2\sigma} u\right \|_{S'(L^2;I)}+\left \|\chi_{B} \nabla  |x|^{-b}|u|^{2\sigma} u\right \|_{S'(L^2;I)}\leq c (T^{\theta_1}+T^{\theta_2})\|\nabla u\|^{2\sigma+1}_{S(L^2;I)},
	\end{align}
	where $I=[-T,T]$. 
\end{lemma}
\begin{proof} 
	We start estimating the first term in the left hand side of \eqref{H1Sl2}. Define $(q_0,p_0)$ the  $L^2$-admissible pair given by\footnote{It is not difficult to check that the pair $(q_0,p_0)$ is $L^2$-admissible.}
	\begin{equation}\label{PA1} 
	q_0= \frac{2(2\sigma+2)}{\sigma(N-2)} \;\;\; \textnormal{and} \;\; p_0=\frac{ N(2\sigma+2)}{ N+ 2\sigma },
	\end{equation}
	From the H\"older inequality and Sobolev inequality, it follows that 
	\begin{align}\label{L1C10}
	\left\|\nabla(|x|^{-b}|u|^{2\sigma} u)\right\|_{L_x^{p'_0}(B^c)} \lesssim & \left\||x|^{-b}\right\|_{L^\gamma(B^c)} \left\|\nabla (|u|^{2\sigma} u) \right\|_{L^{\beta}_x} + \left\|\nabla (|x|^{-b})\right\|_{L^d(B^c)}\|u\|^{2\sigma +1}_{L^{(2\sigma+1)e}_x} \nonumber  \\
	\lesssim &  \left\| |x|^{-b} \right\|_{L^\gamma(B^c)}  \|u\|^{2\sigma}_{L_x^{2\sigma \alpha}} \| \nabla u \|_{L_x^{p_0}} + \left\||x|^{-b-1}\right\|_{L^d(B^c)}\|\nabla  u\|^{2\sigma +1}_{L^{p_0}_x} \nonumber  \\
	\lesssim &  \left\| |x|^{-b} \right\|_{L^\gamma(B^c)} \| \nabla  u \|^{2\sigma+1}_{L_x^{p_0}} + \left\||x|^{-b-1}\right\|_{L^d(B^c)}\|\nabla  u\|^{2\sigma +1}_{L^{p_0}_x},
	\end{align}
	where the following relations are satisfied
	\begin{equation*}
	\left\{\begin{array}{cl}\vspace{0.1cm}
	\frac{1}{p'_0}=&\frac{1}{\gamma}+\frac{1}{\beta}=\frac{1}{d}+\frac{1}{e}\\ \vspace{0.1cm}
	\frac{1}{\beta}=&\frac{1}{\alpha}+\frac{1}{p_0} \\ \vspace{0.1cm}
	1=&\frac{N}{p_0}-\frac{N}{2\sigma \alpha}; \quad p_0<N\\ \vspace{0.1cm}
	1=&\frac{N}{p_0}-\frac{N}{(2\sigma+1)e},
	\end{array}\right.
	\end{equation*}
	which are equivalent to
	\begin{equation}\label{L1C12} 
	\left\{\begin{array}{cl}\vspace{0.1cm}
	\frac{N}{\gamma}=&N-\frac{2N}{p_0}-\frac{2\sigma N}{p_0}+2\sigma \\ \vspace{0.1cm}
	\frac{N}{d}=&N-\frac{2N}{p_0}-\frac{2\sigma N}{p_0} +2\sigma+1.
	\end{array}\right.
	\end{equation}
	Note that, in view of \eqref{PA1} we have $\frac{N}{\gamma}-b=-b<0$ and $\frac{N}{d}-b-1=-b<0$. So, $\left\| |x|^{-b} \right\|_{L^\gamma(B^c)}$ and $\left\| |x|^{-b-1} \right\|_{L^d(B^c)}$ are bounded quantities (see \eqref{RIxb}) and therefore
	\begin{equation*}
	\left\|\nabla(|x|^{-b}|u|^{2\sigma} u)\right\|_{L_x^{p'_0}(B^c)}\lesssim\| \nabla u \|^{2\sigma+1}_{L_x^{p_0}}.
	\end{equation*}
	On the other hand, applying the H\"older inequality in the time variable we deduce
	\begin{equation}\label{C1}
	\left\| \chi_{B^c}\nabla |x|^{-b}|u|^{2\sigma} u   \right\|_{S'\left(L^2;I\right)}\leq \left\|\left\|\nabla(|x|^{-b}|u|^{2\sigma} u)\right\|_{L_x^{p'_0}(B^c)}\right\|_{L_I^{q'_0}}\lesssim T^{\frac{1}{q_1}} \| \nabla u \|^{2\sigma+1}_{L_I^{q_0}L_x^{p_0}},
	\end{equation}
	where $\frac{1}{q'_0}=\frac{1}{q_1}+\frac{2\sigma+1}{q_0}$. From \eqref{PA1}, it is clear that $\frac{1}{q_1}=\frac{4-2\sigma(N-2)}{4}>0$, where the positivity follows from $\sigma<\frac{2-b}{N-2}$. Setting $\theta_1=\frac{1}{q_1}$, we conclude the estimate of the first term in the left hand side of \eqref{H1Sl2}.
	
	Next, we turn out attention to the second term in the left hand side of \eqref{H1Sl2}. First note that $\left\| \chi_B\nabla |x|^{-b}|u|^{2\sigma} u\right\|_{S'\left(L^2;I\right)}\leq \left\|\left\|\nabla(|x|^{-b}|u|^{2\sigma} u)\right\|_{L_x^{p'}(B)}\right\|_{L^{q'}_I}$. From the same arguments as in the inequality \eqref{L1C10}, we deduce
	\begin{equation}\label{L1C221}
	\left\|\nabla(|x|^{-b}|u|^{2\sigma} u)\right\|_{L_x^{p'}(B)}\leq\left \| |x|^{-b} \right\|_{L^\gamma(B)} \| \nabla u \|^{2\sigma+1}_{L_x^{p}} + \left\||x|^{-b-1}\right\|_{L^d(B)}\|\nabla u\|^{2\sigma +1}_{L^{p}_x},
	\end{equation}
	assuming \eqref{L1C12} is satisfied replacing $p_0$ by $p$ (to be determined later), that is
	\begin{equation}\label{L1} 
	\left\{\begin{array}{cl}\vspace{0.1cm}
	\frac{N}{\gamma}=&N-\frac{2N}{p}-\frac{2\sigma N}{p}+2\sigma \\ \vspace{0.1cm}
	\frac{N}{d}=&N-\frac{2N}{p}-\frac{2\sigma N}{p}+2\sigma +1.
	\end{array}\right.
	\end{equation}
	In order to have that $\left\||x|^{-b}\right\|_{L^\gamma(B)}$ and $\left\||x|^{-b-1}\right\|_{L^d(B)}$ are bounded quantities, we need $\frac{N}{\gamma}>b$ and $\frac{N}{d}>b+1$, respectively, by \eqref{RIxb}. So, we want to show that $N-\frac{2N}{p}-\frac{2\sigma N}{p}+2\sigma>b$. This is equivalent to $2\sigma<\frac{(N-b)p-2N}{N-p}$ (assuming $p<N$), then we choose $p$ such that 
	$$
	\frac{(N-b)p-2N}{N-p}=\frac{4-2b}{N-2}.
	$$
	In other words, we choose $p$ and $q$ given by\footnote{It is easy to see that $p>2$ if, and only if, $N>2$ and $p<\frac{2N}{N-2}$ if, and only if, $b<2$. Therefore the pair $(q,p)$ is $L^2$-admissible.} 
	\begin{equation}\label{PA2}
	p=\frac{2N(N-b)}{N(N-2)+4-bN}\;\;\textnormal{and}\;\;q=\frac{2(N-b)}{N-2},
	\end{equation}
	where we have used that the pair $(q,p)$ is $L^2$-admissible to compute the value of $q$. Note that $p<N$ if, and only if, $b<N-2$. Hence, the H\"older inequality in the time variable leads to
	\begin{align}
	\left\| \chi_B\nabla |x|^{-b}|u|^{2\sigma} u\right\|_{S'\left(L^2;I\right)}&\lesssim T^{\frac{1}{q_1}} \|\nabla u\|^{2\sigma+1}_{L_I^{q}L^p_x},  
	\end{align}  
	where $\frac{1}{q'}=\frac{1}{q_1}+\frac{2\sigma+1}{q}$. Since $\sigma<\frac{2-b}{N-2}$, it is clear that $\frac{1}{q_1}=1-\frac{2\sigma+2}{q}=\frac{4-2b-2\sigma(N-2)}{2(N-b)}>0$. Therefore,
	\begin{align}\label{c2}
	\left\| \chi_B\nabla |x|^{-b}|u|^{2\sigma} u\right\|_{S'\left(L^2;I\right)}\lesssim T^{\theta_2} \|\nabla u\|^{2\sigma+1}_{S(L^2;I)},   
	\end{align}
	for $\theta_2=\frac{1}{q_1}>0$.
	
	Note that the restriction $0<b<N-2$ implies that the inequality \eqref{c2} only holds for $0<b<1$ when $N=3$. Next, we show that in dimension $N=3$ it is also possible to  consider $1\leq b <\frac{3}{2}$. To this end, for a pair $(q,p)$ $L^2$-admissible to be chosen later, we have
	\begin{align}
	\left\| \chi_B\nabla |x|^{-b}|u|^{2\sigma} u\right\|_{S'\left(L^2;I\right)}&\leq \left\|\chi_{B}\nabla |x|^{-b}|u|^{2\sigma}u\right\|_{L^{q'}_IL^{p'}_x}\\
	&\leq \left\|\left\||x|^{-b}\right\|_{L^{\gamma}(B)}\|u\|_{L^{2\sigma r_1}_x}^{2\sigma}\left\|\nabla u\right\|_{L ^{\bar p}_x}\right\|_{L^{q'}_I}+\left\|\left\||x|^{-b-1}\right\|_{L^{d}(B)}\|u\|_{L^{(2\sigma+1)e}_x}^{2\sigma+1}\right\|_{L^{q'}_I}
	\nonumber\\&\lesssim T^{\frac{1}{q_1}}\left\||x|^{-b}\right\|_{L^{\gamma}(B)}\|\nabla u\|_{L^{\bar q}_IL^{\bar p}_x}^{2\sigma+1}+T^{\frac{1}{q_1}}\left\||x|^{-b-1}\right\|_{L^{d}(B)}\|\nabla u\|_{L^{\bar q}_IL^{\bar p}_x}^{2\sigma+1},
	\end{align}
if the following conditions are satisfied
	\begin{align}\label{condH1sl2}
	\begin{cases}
	\frac{1}{p'}=\frac{1}{\gamma}+\frac{1}{r_1}+\frac{1}{\bar p}=\frac{1}{d}+\frac{1}{e}\\
	1=\frac{3}{\bar p}-\frac{3}{2\sigma r_1}=\frac{3}{\bar p}-\frac{3}{(2\sigma+1)e},\,\,\,\,\bar p<3\\
	\frac{1}{q'}=\frac{1}{q_1}+\frac{2\sigma+1}{\bar q}.
	\end{cases}
	\end{align}
	Consider the $L^2$-admissible pair $(\bar q,\bar p )$ given by
	\begin{align}
	\bar q=\frac{1-2\varepsilon}{2}\,\,\, \mbox{and}\,\,\,\,\bar p=\frac{3}{1+\varepsilon},
	\end{align}
	for $\varepsilon>0$ small enough. In this case, we have $\bar p<3$ and system \eqref{condH1sl2} can be rewritten  as 
	\begin{align}
	\begin{cases}
	\frac{3}{\gamma}-b=2-b+2\sigma-\frac{3}{p}-\varepsilon(2\sigma+1)\\
	\frac{3}{d}-b-1=2-b+2\sigma-\frac{3}{p}-\varepsilon(2\sigma+1).
	\end{cases}
	\end{align}
	Now, we need to choose $p$ such that $2-b+2\sigma-\frac{3}{p}-\varepsilon(2\sigma+1)>0$, and thus, by \eqref{RIxb}, we deduce that $\left\||x|^{-b}\right\|_{L^\gamma (B)}$ and $\left\||x|^{-b-1}\right\|_{L^d (B)}$ are bounded quantities. To this end, define the $L^2$-admissible $(q,p)$ given by\footnote{Since $1\leq b<\frac{3}{2}$ it is clear that $2<p<6$ and $q>0$ for $\varepsilon>0$ small enough. }
	\begin{align}
	q=\frac{4}{2b-1+4\varepsilon (2\sigma+1)}\,\,\,\,\mbox{ and }\,\,\,\,p=\frac{3}{2-b-2\varepsilon(2\sigma+1)}.
	\end{align}
	It also follows that, $\frac{1}{q_1}=1-\frac{1}{q}-\frac{2\sigma+1}{\bar q}=1-\frac{2b-1+4\varepsilon(2\sigma+1)}{4}-\frac{(2\sigma+1)(1-2\varepsilon)}{2}=\frac{4-2b-2\sigma-2\varepsilon(2\sigma+1)}{4}>0$ for $\varepsilon>0$ sufficient small. Therefore,
	\begin{align}
	\left\| \chi_B\nabla |x|^{-b}|u|^{2\sigma} u\right\|_{S'\left(L^2;I\right)}\lesssim T^{\theta_2}\|\nabla u\|_{S(L^2)},
	\end{align}
	for $\theta_2=\frac{1}{q_1}>0$. 
	
	Finally, collecting the last inequality, \eqref{C1} and \eqref{c2} we conclude the proof.
\end{proof}

\begin{lemma}\label{lema2} Let $N\geq 3 $, $0<b<\min\{\frac{N}{2},2\}$ and $\frac{2-b}{N}<\sigma<\frac{2-b}{N-2}$, then there exist $c,\theta_1,\theta_2 >0$ such that the following inequalities hold
	\begin{align}
	\left \|D^{s_c} |x|^{-b}|u|^{2\sigma}u\right \|_{S'(L^2;I)}\leq \left\|D^{s_c} |x|^{-b}|u|^{2\sigma}u\right\|_{L^2_IL^{\frac{2N}{N+2}}_x}\leq c (T^{\theta_1}+T^{\theta_2})\|\nabla u\|_{S(L^2;I)}\|D^{s_c} u\|^{2\sigma}_{S(L^2;I)},
	\end{align}
	where $I=[-T,T]$.
\end{lemma}

\begin{proof}		
	The first inequality comes from the fact that the pair $\left(2,\frac{2N}{N-2}\right)$ is $L^2$-admissible. Now, from Sobolev embedding (see Lemma \ref{SI} $(i)$) we obtain
	\begin{align}\label{sobolev}
	\left\|D^{s_c} |x|^{-b}|u|^{2\sigma}u\right\|_{L^\frac{2N}{N+2}_x}\lesssim \left\|\nabla |x|^{-b}|u|^{2\sigma}u\right\|_{L^{p^*}_x},
	\end{align} 
	where $p^*=\frac{2N\sigma}{4\sigma+2-b}$. 
	
	Let $A\subset \mathbb{R}^N$ denotes either $B$ or $B^C$. Applying H\"older's inequality first in space and then in time, we get
	\begin{align}\label{D}
	\left\|\nabla |x|^{-b}|u|^{2\sigma}u\right\|_{L^{2}_IL_x^{p^*}(A)} \lesssim & \left\| \left\||x|^{-b}\right\|_{L^\gamma(A)} \|u\|_{L^{2\sigma \beta}_x}^{2\sigma} \|\nabla u\|_{L^p_x} + \left\||x|^{-b-1}\right\|_{L^d(A)}\|u\|^{2\sigma}_{L^{2\sigma e}_x}\|u\|_{L^{f}_x}\right\|_{L^{2}_I} \nonumber  \\
	\lesssim & \left\|  \left\| |x|^{-b} \right\|_{L^\gamma(A)}  \|D^{s_c}u\|^{2\sigma}_{L^p_x}   \|\nabla u \|_{L_x^{p}} + \left\||x|^{-b-1}\right\|_{L^d(A)}\|D^{s_c}u\|^{2\sigma}_{L^{p}_x}\|\nabla u\|_{L^{p}_x}\right\|_{L^{2}_I} \nonumber  \\
	\lesssim & T^{\frac{1}{q^*}}\left\||x|^{-b} \right\|_{L^\gamma(A)}  \|D^{s_c}u\|^{2\sigma}_{L^q_IL^p_x}   \|\nabla u \|_{L^q_IL_x^{p}}\nonumber \\
	&\qquad + T^{\frac{1}{q^*}}\left\||x|^{-b-1}\right\|_{L^d(A)}\|D^{s_c}u\|^{2\sigma}_{L^q_IL^{p}_x}\|\nabla u\|_{L^q_IL^{p}_x},
	\end{align}
	if the following conditions are satisfied 
	\begin{equation}\label{B1a} 
	\left\{\begin{array}{cl}\vspace{0.1cm}
	\frac{1}{p^*}=&\frac{1}{\gamma}+\frac{1}{\beta}+\frac{1}{p}=\frac{1}{d}+\frac{1}{e}+\frac{1}{f},\\ \vspace{0.1cm} 
	s_c=&\frac{N}{p}-\frac{N}{2\sigma \beta}=\frac{N}{p}-\frac{N}{2\sigma e},\;\;\;1=\frac{N}{p}-\frac{N}{f},\;\;\;p<N\\
	\frac{1}{2}=&\frac{1}{q^*}+\frac{2\sigma}{q}+\frac{1}{q}.
	\end{array}\right.
	\end{equation}
	The above conditions are equivalent to
	\begin{align}\label{sistema} 
	\left\{\begin{array}{cl}\vspace{0.1cm}
	\frac{N}{\gamma}-b&=\frac{N}{p^*}+2\sigma s_c-\frac{N(2\sigma +1)}{p}-b,\,\,p<N \\ \vspace{0.1cm}
	\frac{N}{d}-b-1&=\frac{N}{p^*}+2\sigma s_c-\frac{N(2\sigma +1)}{p}-b \\
	\frac{1}{q^*}&=\frac{1}{2}-\frac{2\sigma+1}{q}.
	\end{array}\right.
	\end{align}
	Our goal is to find a pair $(q,p)$ $L^2$-admissible such that  $\left\||x|^{-b}\right\|_{L^\gamma(A)}$ and $\left\||x|^{-b-1}\right\|_{L^d(A)}$ are bounded quantities (see \eqref{RIxb}), $p<N$ and $\frac{1}{q^*}>0$. Let $(q_{\pm},p_{\pm})$ defined by
	\begin{align}\label{ppm}
	p_{\pm}=\frac{2\sigma(2\sigma+1)N}{2\sigma^2N+2-b\pm \varepsilon}\,\,\,\,\mbox{ and }\,\,\,\,q_{\pm}=\frac{4\sigma(2\sigma+1)}{\sigma N-2+b\mp \varepsilon},
	\end{align}
	for $\varepsilon>0$ sufficiently small. Note that $\frac{2}{q_{\pm}}=\frac{N}{2}-\frac{N}{p_{\pm}}$ and $p_{\pm},q_{\pm}>0$ whenever $\frac{2-b}{N}<\sigma <\frac{2-b}{N-2}$.
	Furthermore, $2<p_{\pm}<\frac{2N}{N-2}$ and $p_{\pm}<N$.
	Indeed, since $\sigma(2\sigma+1)N>2\sigma^2N+2-b\pm\varepsilon$ and $\varepsilon>0$ small enough, we get $p_{\pm}>2$.  As $-\sigma (N-2)+2-b\pm \varepsilon>0$ for $\varepsilon$ sufficiently small, then $-4\sigma^2<-\sigma(N-2)+2-b\pm \varepsilon$. The last inequality is equivalent to $\sigma(2\sigma+1)(N-2)<2\sigma^2N+2-b\pm \varepsilon$, which implies $p_{\pm}<\frac{2N}{N-2}$. If $N\geq 4$, then $N>\frac{2N}{N-2}>p_{\pm}$. When $N=3$, we have $p_{\pm}<3$ if, and only if, $0<2\sigma^2-2\sigma+2-b\pm \varepsilon$ which is true for $\varepsilon$ small enough and $0<b<\frac{3}{2}$. Summing up, 
	\begin{align}
	(q_{\pm},p_{\pm})\mbox{ is }L^2\mbox{-admissible \,\,\,\, and \,\,\,\,} p_{\pm}<N.
	\end{align}
	In addition,
	\begin{align}
	\frac{1}{q^*_{\pm}}=\frac{1}{2}-\frac{2\sigma+1}{q_{\pm}}=\frac{1}{2}-\frac{\sigma-2+b\mp \varepsilon}{4\sigma}=\frac{-\sigma(N-2)+2-b\pm \varepsilon}{4\sigma}>0.
	\end{align}
	Now, if $A=B^c$ we choose $(q,r)=(q_{+},r_{+})$ and $\theta_1=\frac{1}{q^*_{+}}$. Then,  $\frac{N}{\gamma}-b<0$ and $\frac{N}{d}-b-1<0$, and consequently, $\left\||x|^{-b}\right\|_{L^\gamma(B^c)},\;\left\||x|^{-b-1}\right\|_{L^d(B^c)}<\infty$ .
	On the other hand, if $A=B$ we choose the pair $(q,r)=(q_-,p_-)$ and $\theta_2=\frac{1}{q^*_{-}}$, so we also get $\left\||x|^{-b}\right\|_{L^\gamma(B)}, \left\||x|^{-b-1}\right\|_{L^d(B)}<\infty$. 
	Finally, the relations \eqref{sobolev} and \eqref{D} imply the desired result.
\end{proof}

\begin{lemma}\label{lemmacontraction} Let $N\geq 3 $, $0<b<2$ and $\frac{2-b}{N}<\sigma<\frac{2-b}{N-2}$, then there exist $c,\theta_1,\theta_2 >0$ and $\theta\in (0,2\sigma) $ small enough such that the following inequality holds  
	\begin{align}
	\left \|\chi_{B^c} |x|^{-b}|u|^{2\sigma} v\right \|_{S'(\dot H^{-s_c};I)}+&\left\|\chi_{B} |x|^{-b}|u|^{2\sigma} v\right \|_{S'(\dot H^{-s_c};I)}\\
	&\quad\quad\leq c (T^{\theta_1}+T^{\theta_2})\|\nabla u\|^{\theta}_{L^{\infty}_IL^2_x}\|u\|^{2\sigma-\theta}_{S(\dot H^{s_c};I)}\|v\|_{S(\dot H^{s_c};I)},
	\end{align}
where $I=[-T,T]$.
\end{lemma}
\begin{proof}
	Let $(\tilde a,r)$ a pair $\dot H^{-s_c}$-admissible and $A\subset \mathbb{R}^N$ denotes either $B$ or $B^C$. As in the previous lemmas, an application of the H\"older inequality first in space and then in time yields
	
	\begin{align}\label{E}	
	\left\||x|^{-b}|u|^{2\sigma}v \right\|_{L^{\tilde a'}_IL_x^{r'}(A)} &\leq \left\|  \left\| |x|^{-b} \right\|_{L^\gamma(A)}  \|u\|^{\theta}_{L_x^{\theta r_1}} \|u\|^{2\sigma-\theta}_{L_{x}^{r}}  \| v \|_{L_x^{r}}\right\|_{L^{\tilde a'}_I} \nonumber  \\
	&\lesssim  \left\||x|^{-b}\right\|_{L^{\gamma}(A)}\left\| \| \nabla u \|^{\theta}_{L_x^{2}}\|u\|^{2\sigma-\theta}_{L^r_x}\| v \|_{L_x^{r}}  \right\|_{L^{\tilde a'}_I}\nonumber \\
	&\lesssim  T^{\frac{1}{q_1}}\left\||x|^{-b}\right\|_{L^{\gamma}(A)}\| \nabla u \|^{\theta}_{L_t^{\infty}L_x^{2}}\|u\|^{2\sigma-\theta}_{L^a_tL^r_x}\| v \|_{L_t^{a}L_x^{r}},
	\end{align}
	assuming the following relations hold
	\begin{equation}\label{E1}
	\left\{\begin{array}{rcl}\vspace{0.1cm}
	\frac{1}{r'}&=&\frac{1}{\gamma}+\frac{1}{r_1}+\frac{2\sigma-\theta}{r}+\frac{1}{r},\\ \vspace{0.1cm}
	1&=&\frac{N}{2}-\frac{N}{\theta r_1}, \\ \vspace{0.1cm}
	\frac{1}{\tilde a'}&=&\frac{1}{q_1}+\frac{2\sigma-\theta}{a}+\frac{1}{a},
	\end{array}\right.
	\end{equation}
	for $\theta\in (0,2\sigma)$ small enough. 
	
	If the pair $(a,r)$ is $\dot H^{s_c}$-admissible (so, $\frac{1}{\tilde a}-\frac{1}{a}=s_c$), the conditions \eqref{E1} are equivalent to
	\begin{equation}\label{ADMREL}
	\left\{\begin{array}{rcl}\vspace{0.1cm}
	\frac{N}{\gamma}-b&=&N-b-\frac{N\theta}{2}+\theta-\frac{N(2\sigma +2-\theta)}{r}\\ \vspace{0.1cm}
	\frac{1}{q_1}&=&1-s_c-\frac{2\sigma+2-\theta}{a}.
	\end{array}\right.
	\end{equation} 
	Now, we shall choose $(a,r)$ satisfying \eqref{ADMREL}, $\frac{1}{q_1}>0$ and $\frac{N}{\gamma}-b>0$, if $A=B$ or $\frac{N}{\gamma}-b<0$, if $A=B^c$ (see \eqref{RIxb})\footnote{Note that if we find this pair then we also find a pair $(\tilde a,r)$ $\dot H^{-s_c}$-admissible using the relation $\frac{1}{\tilde a}-\frac{1}{a}=s_c$.}. We first treat the case $A=B$ and define
	\begin{align}
	a=\frac{2\sigma+2-\theta}{1-s_c-\varepsilon}\,\,\,\,\mbox{ and }\,\,\,\,r=\frac{N\sigma_c(2\sigma+2-\theta)}{N(2\sigma+2-\theta)-2\sigma_c(1-s_c-\varepsilon)},
	\end{align}
	for $0<\varepsilon<\frac{\theta(1-s_c)}{2}$ small enough. 
	Thus, $(a,r)$ is a $\dot H^{s_c}$-admissible\footnote{Since $a>\frac{2}{1-s_c}$ and $\frac{2}{a}=\frac{N}{\sigma_c}-\frac{N}{r}$, we have $r<\frac{2N}{N-2}$. On the other hand,  $r>\sigma_c$ if, and only if, $0<\varepsilon<1-s_c$.} and
	\begin{align}
	\frac{1}{q_1}&=1-s_c-\frac{2\sigma+2-\theta}{a}=\varepsilon>0,\\
	\frac{N}{\gamma}-b&=N-b-\frac{N\theta}{2}+\theta-\frac{N(2\sigma+2-\theta)}{\sigma_c}+2(1-s_c-\varepsilon)=\theta(1-s_c)-2\varepsilon>0.	
	\end{align}
	
	Next, we consider $A=B^c$ and define the following numbers
	\begin{align}
	a=\infty\,\,\,\,\mbox{ and }r=\frac{2N}{N-2s_c}=\frac{2\sigma N}{2-b}.
	\end{align}  
	It is not difficult to see that $(a,r)$ is $\dot H^{s_c}$-admissible and
	\begin{align}
	\frac{1}{q_1}&=1-s_c>0,\\ \frac{N}{\gamma}-b&=N-b-\frac{N\theta}{2}+\theta-\frac{(2\sigma+2-\theta)(2-b)}{2\sigma}=-(2-\theta)(1-s_c)<0.
	\end{align}
	This complete the proof of Lemma \ref{lemmacontraction}.
\end{proof}

Now, with the previous lemmas in hand we are in a position to prove Theorem \ref{LWP}.

\begin{proof}[Proof of Theorem \ref{LWP}]	
For any ($q,p$) $L^2$-admissible and $(a,r)$ $\dot H^{s_c}$-admissible, we set (recall definition \eqref{StrNorm}) 
$$
X=  \left( \bigcap_{(q,p)\in \mathcal{A}_0}L^q\left([-T,T];\dot{H}^{s_c,p}\cap \dot{H}^{1,p}\right)\right)\bigcap \left( \bigcap_{(a,r)\in \mathcal{A}_{s_c}}L^a\left([-T,T]; L^r \right)\right)$$ and 
\begin{equation*}\label{NHs} 
\|u\|_T=\|\nabla u\|_{S\left(L^2;I\right)}+\|D^{s_c} u\|_{S\left(L^2;I\right)}+\|u\|_{S\left(\dot H^{s_c};I\right)},
\end{equation*}
where $I=[-T,T]$.

For $m,T>0$, define the set
\begin{equation*}
S(m,T)=\{u \in X : \|u\|_T\leq m \}
\end{equation*}
with the metric 
$$
d_T(u,v)=\|u-v\|_{S\left(\dot H^{s_c};I\right)}.
$$
In Appendix \ref{appA} we prove that $(S(m,T), d_T)$ is a complete metric space.

We shall show that $G=G_{u_0}$ defined by the right hand side of \eqref{duhamel} 
is a contraction on $(S(m,T), d_T)$ for a suitable choice of $m$ and $T$. Indeed, it follows from the Strichartz inequalities in Lemma \ref{Lemma-Str} that 
\begin{align}\label{NSD} 
\|\nabla G(u)\|_{S\left(L^2;I\right)}&\leq c\|\nabla u_0\|_{L^2}+\left\|\nabla \int_{0}^{t} e^{i(t-t')\Delta} |x|^{-b}|u|^{2\sigma} u(t')\,dt' \right\|_{S\left(L^2;I\right)}\\
&\leq c\|\nabla u_0\|_{L^2}+c\left\|\chi_{B}\nabla |x|^{-b}|u|^{2\sigma} u\right\|_{S'\left(L^2;I\right)}+c\left\|\chi_{B^c}\nabla |x|^{-b}|u|^{2\sigma} u\right\|_{S'\left(L^2;I\right)},
\end{align}
\begin{align}
\|D^{s_c} G(u)\|_{S\left(L^2;I\right)}&\leq c \|D^{s_c} u_0\|_{L^2}+ c\|D^{s_c} |x|^{-b}|u|^{2\sigma} u\|_{S'\left(L^2;I\right)},
\end{align}
and
\begin{align}
\|G(u)\|_{S\left(\dot H^{s_c};I\right)}\leq c\|u_0\|_{\dot H^{s_c}}+c\|\chi_{B}|x|^{-b}|u|^{2\sigma} u\|_{S'\left(\dot H^{-s_c};I\right)}+c\|\chi_{B^c}|x|^{-b}|u|^{2\sigma} u\|_{S'\left(\dot H^{-s_c};I\right)}.
\end{align}
So, applying Lemmas \ref{lema1}-\ref{lemmacontraction} we deduce
\begin{align}\label{F} 
\|\chi_{B}\nabla |x|^{-b}|u|^{2\sigma} u\|_{S'(L^2;I)}+\|\chi_{B^c}\nabla |x|^{-b}|u|^{2\sigma} u\|_{S'(L^2;I)}&\leq c(T^{\theta_1}+T^{\theta_2}) \| \nabla u \|^{2\sigma+1}_{S(L^2;I)},
\end{align}
\begin{align}
\|D^{s_c} |x|^{-b}|u|^{2\sigma} u\|_{S'(L^2;I)}&\leq c(T^{\theta_1}+T^{\theta_2}) \| \nabla u \|_{S(L^2;I)}\|D^{s_c}u\|^{2\sigma}_{S(L^2;I)},
\end{align}
and
\begin{align}
\|\chi_{B} |x|^{-b}|u|^{2\sigma} u\|_{S'(\dot H^{-s_c};I)}+\|\chi_{B^c}|x|^{-b}|u|^{2\sigma} u\|_{S'(\dot H^{-s_c};I)}&\leq c(T^{\theta_1}+T^{\theta_2}) \| \nabla u \|^{\theta}_{L^{\infty}_tL^2_x}\|u\|^{2\sigma+1-\theta}_{S(\dot H^{s_c};I)},
\end{align}
for some $\theta_1,\theta_2>0$. Hence, if $u \in S(m,T)$ then
\begin{align}\label{esa}
\|G(u)\|_T 
&\leq c \|u_0\|_{\dot{H}^{s_c}\cap \dot{H}^1 }+c (T^{\theta_1}+T^{\theta_2}) m^{2\sigma+1}.
\end{align}
Now, choosing $m\geq 2c\|u_0\|_{\dot{H}^{s_c}\cap \dot{H}^1 }$ and $T>0$ such that 
\begin{equation}\label{CTHs} 
c (T^{\theta_1}+T^{\theta_2})m^{2\sigma} < \frac{1}{4},
\end{equation}
we obtain $G(u)\in S(m,T)$. Such calculations establishes that $G$ is well defined on $S(m,T)$.
To prove that $G$ is a contraction we first recall the elementary inequality
\begin{equation}\label{nonlinearity}
||x|^{-b}|u|^{2\sigma} u-|x|^{-b}|v|^{2\sigma} v|\lesssim |x|^{-b}\left( |u|^{2\sigma}+ |v|^{2\sigma} \right)|u-v|.
\end{equation}
 Then, an application of Lemma \ref{lemmacontraction} yields
\begin{align*}
d_T(G(u),G(v))&\leq  c (T^{\theta_1}+T^{\theta_2})\left(\|\nabla u\|^{\theta}_{L^{\infty}_tL^2_x}\|u\|_{S(\dot H^{s_c};I)}^{2\sigma-\theta}+\|\nabla v\|^{\theta}_{L^{\infty}_tL^2_x}\|v\|_{S(\dot H^{s_c};I)}^{2\sigma-\theta}\right)\|u-v\|_{S(\dot H^{s_c};I)}\\
&\leq c (T^{\theta_1}+T^{\theta_2})\left(\|u\|^{2\sigma}_T+\|v\|^{2\sigma}_T\right)d_T(u,v),
\end{align*}
and so, taking $u,v\in S(m,T)$ we get
$$
d_T(G(u),G(v))\leq c (T^{\theta_1}+T^{\theta_2})m^{2\sigma} d_T(u,v).
$$
Therefore, from \eqref{CTHs}, $G$ is also a contraction on $S(m,T)$. Finally, by the contraction mapping principle we have a unique $u \in S(m,T)$ such that $G(u)=u$ and the proof is completed.
\end{proof}

Let $T^*=T^*(u_0)>0$ be the maximal positive time of existence for a solution $u$ to \eqref{PVI} in $\dot H^{s_c}(\mathbb{R}^N)\cap \dot H^1(\mathbb{R}^N)$ given by Theorem \ref{LWP}. If $T^*=\infty$, we say that the solution is global. On the other hand if $T^*<\infty$, as consequence of the proof of Theorem \ref{LWP}, we get the following blow-up alternative and a lower bound on the blow-up rate.
\begin{coro}\label{Blowalt}
	Let $N\geq 3$ and $0<b<\min\{\frac{N}{2},2\}$. If $u$ is a solution to the IVP \eqref{PVI} with finite maximal positive time of existence $0<T^*<\infty$, then $\displaystyle\lim_{t\to T^*}\|u(t)\|_{\dot H^{s_c}\cap \dot H^1}=\infty$. Moreover, there exist $c,\widetilde\theta_1, \widetilde\theta_2>0$ such that
	\begin{align}\label{bounded}
	\|u(t)\|_{\dot H^{s_c}\cap \dot H^1}> \frac{c}{(T^*-t)^{\widetilde\theta_1}+(T^*-t)^{\widetilde\theta_2}},\quad \textit{for all} \quad t\in [0,T^*).
	\end{align}
\end{coro}

\begin{proof}
	Assume that there exist $0<m<\infty$ and a sequence $\{t_n\}_{n\in \mathbb{N}}$ with $t_n\uparrow T^*$ such that $\|u(t_n)\|_{\dot H^{s_c}\cap \dot H^1}\leq m $ for all $n\geq 1$. Let $T(m)$ denote the existence time obtained by Theorem \ref{LWP} for all initial data bounded above by $m$ and $k\in \mathbb{N}$ such that $t_k+T(m)>T^*$. By \eqref{CTHs} and starting from $u(t_k)$, one can extend $u$ up to $t_k+T(m)$, which contradicts the maximality of $T^*$, and thus,
	\begin{align}
	\|u(t)\|_{\dot H^{s_c}\cap\dot H^1}\to \infty,\,\,\,\,\mbox{ as }t\uparrow T^*.
	\end{align}
	
Moreover, it follows from \eqref{esa} and the fixed point argument that if for some $m>0$,
	\begin{align}
	c\|u(t)\|_{\dot H^{s_c}\cap \dot H^1}+c\left[(\tau-t)^{\theta_1}+(\tau-t)^{\theta_2}\right]m^{2\sigma+1}\leq m,
	\end{align}  
	then $\tau<T^*$. Thus, 
	\begin{align}
	c\|u(t)\|_{\dot H^{s_c}\cap \dot H^1}+c\left[(T^*-t)^{\theta_1}+(T^*-t)^{\theta_2}\right]m^{2\sigma+1}> m,
	\end{align}  
	for all $m>0$. Choosing $m=2c\|u(t)\|_{\dot H^{s_c}\cap \dot H^1}$, it follows that 
	\begin{align}
	\left[(T^*-t)^{\theta_1}+(T^*-t)^{\theta_2}\right]\|u(t)\|_{\dot H^{s_c}\cap \dot H^1}^{2\sigma}> c.
	\end{align}
	In particular, we have
	\begin{align}
	\|u(t)\|_{\dot H^{s_c}\cap \dot H^1}> \frac{c}{(T^*-t)^{\widetilde\theta_1}+(T^*-t)^{\widetilde\theta_2}},
	\end{align}
	with $\widetilde \theta_i=\frac{\theta_i}{2\sigma}$ for $i=1,2$, completing the proof.
\end{proof}

\begin{rem}\label{eqtime}
	It is possible to derive a more precise lower bound on the blow-up rate for type II blow-up solutions. Let $u\in C([0,T^*); \dot H^{s_c}\cap \dot H^1 )$ be a solution to the IVP \eqref{PVI} with finite maximal positive time of existence $0<T^*<\infty$. If we assume the condition
	\begin{align}
	\sup_{t\in [0,T^*)}\|u(t)\|_{\dot H^{s_c}}=M<\infty,
	\end{align} 
then from the local well-posedness theory in $\dot H^{s_c}(\mathbb{R}^N)\cap \dot H^1(\mathbb{R}^N)$, we deduce the following lower bound for the blow-up rate
	\begin{align}\label{lowerbounded}
	\|\nabla u(t)\|_{L^2}\geq \frac{c}{(T^*-t)^{\frac{1-s_c}{2}}}, \quad \textit{for all} \quad t\in [0,T^*). 
	\end{align}
	Indeed, for $t\in [0,T^*)$ we consider the following scaling of $u$
	\begin{align}
	v^t(x,\tau)=\rho^{\frac{2-b}{2\sigma}}(t)u(\rho(t)x,t+\rho^2(t)\tau)
	\end{align}
	where $\rho(t)^{1-s_c}\|\nabla u(t)\|_{L^2}=1$. Hence, $v^t(0)\in \dot H^{s_c}(\mathbb{R}^N)\cap \dot H^1(\mathbb{R}^N)$ and, by simple computations, we can find $m>0$, such that $\|v^t(0)\|_{\dot H^{s_c}\cap \dot H^1}\leq m$ for all $t\in [0,T^*)$. Thus, from the local theory in $\dot H^{s_c}(\mathbb{R}^N)\cap \dot H^1(\mathbb{R}^N)$, there exists $\tau_0$, independent of $t$, such that $v^t$ is defined on $[0,\tau_0]$. Then, $t+\rho^2(t)\tau_0<T^*$, and consequently we obtain \eqref{lowerbounded}.
	
\end{rem}

\section{Gagliardo-Nirenberg inequality and Global solutions}\label{sec4}
As we mentioned in the introduction, using a Sobolev embedding (see Stein-Weiss \citep[Theorem B*]{stein}), Campos and the first author \cite{campos2018critical} established the following Gagliardo-Nirenberg type inequality for functions $f\in \dot H^1(\mathbb{R}^N)\cap L^{\sigma_c}(\mathbb{R}^N)$
\begin{equation}\label{GNcrit}
\int |x|^{-b}|f|^{2\sigma+2} \, dx\leq c \|\nabla f\|^2_{L^2}\|f\|^{2\sigma}_{L^{\sigma_c}},
\end{equation}
where $N\geq 2$, $0<b<2$, $\sigma$ in the intercritical regime ($\frac{2-b}{N}<\sigma<\frac{2-b}{N-2}$, if $N\geq 3$ or $\frac{2-b}{N} < \sigma < \infty$, if $N=2$) and $\sigma_c = \frac{2N\sigma}{2-b}$.
In this section, we investigate the sharp constant for inequality above. As a consequence we also prove that this inequality holds when $N=1$ (with the same restrictions on the other parameters).
\subsection{The ground states}
We first recall that for $N\geq 1$, $0<b<2$ and $\sigma$ in the intercritical regime the second author \cite{Farah_well}, following the ideas introduced by \citet{W_Nonl}, obtained the following Gagliardo-Nirenberg type inequality
\begin{align}
\int |x|^{-b}|f|^{2\sigma+2}\,dx\leq C_{GN}\|\nabla f\|_{L^2}^{2\sigma s_c+2}\|f\|_{L^2}^{2\sigma(1-s_c)}
\end{align}
with the sharp constant $C_{GN}>0$ given explicitly by
\begin{align}
C_{GN}=\left[\frac{2\sigma(1-s_c)}{2\sigma s_c+2}\right]^{\sigma s_c}\frac{2\sigma+2}{(2\sigma s_c+2)\|Q\|_{L^2}^{2\sigma}},
\end{align}
where $Q$ is the unique radially-symmetric, positive, decreasing solution of the elliptic problem 
\begin{align}\label{criticalgroundstate}
\Delta Q+|x|^{-b}|Q|^{2\sigma}Q=Q.
\end{align}
The proof relies mainly on the fact that the functional $f\mapsto\displaystyle\int |x|^{-b}|f|^{2\sigma+2}$ is weakly sequentially continuously (see \citet[Lemma 2.1]{genoud2012critical} and the references therein),
Here we follow a similar approach to study the inequality \eqref{GNcrit} and, as we will see below, the sharp constant is directly connected with the solutions of the  elliptic equation
\begin{align}\label{groundstate}
\Delta \phi +|x|^{-b}|\phi|^{2\sigma}\phi=|\phi|^{\sigma_c-2}\phi.
\end{align}
In the following lemma we obtain two Pohozaev-type identities which are satisfied by any solution  of \eqref{groundstate}.
\begin{lemma}\label{pohozaev}
		Let $N\geq 1$, $0<b<2$, $\frac{2-b}{N}<\sigma<\frac{2-b}{N-2}$ ($\frac{2-b}{N}<\sigma<\infty$, if  $N=1,2$) and $\sigma_c=\frac{2N\sigma}{2-b}$. Let $\phi\in \dot H^1(\Real^N)\cap L^{\sigma_c}(\Real^N)$ be a solution of \eqref{groundstate}. Then the following identities hold
	\begin{align}\label{h1sc}
	\int |\nabla \phi|^2\,dx=\frac{1}{\sigma}\int |\phi|^{\sigma_c}\,dx
	\end{align}
	\begin{align}\label{epsc}
	\int |x|^{-b}|\phi|^{2\sigma+2}\,dx=\frac{\sigma+1}{\sigma}\int |\phi|^{\sigma_c}\,dx.
	\end{align}	
\end{lemma}
\begin{proof}
	The proof of these identities is classical and we provide the details for the reader's convenience. Multiplying the equation \eqref{groundstate} by $x\cdot \nabla \bar \phi$ and taking the real part, we obtain
	\begin{align}\label{l1}
	\mbox{Re}\,\int \Delta \phi\, x\cdot \nabla \bar \phi\,dx+\mbox{Re}\,\int |x|^{-b}|\phi|^{2\sigma}\phi\, x\cdot \nabla \bar \phi\,dx=\mbox{Re}\,\int |\phi|^{\sigma_c-2}\phi\, x\cdot \nabla \bar \phi\,dx.
	\end{align}
	We consider the first term in the left hand side of \eqref{l1}. Indeed, integrating by parts, we get
	\begin{align}
	\int \Delta \phi\,x\cdot \nabla \bar{\phi}\,dx&=\sum_{i,j=1}^N\int \partial^2_{i}\phi x_j\partial_j\bar \phi\,dx=-\sum_{i,j=1}^N\int \partial_{i}\phi\partial_i(x_j\partial_j\bar \phi)\,dx\\
	&=-\sum_{i=1}^N\int \partial_{i}\phi\partial_i\bar \phi\,dx-\sum_{i,j=1}^N\int \partial_i \phi x_j\partial^2_{ij}\bar \phi\,dx\\
	&=-\int |\nabla \phi|^2\,dx-\sum_{i,j=1}^N\int 
	\phi(\delta_{ij}\partial^2_{ij}\bar \phi+x_j\partial^3_{iij}\bar \phi)\,dx\\
	&=-2\int |\nabla \phi|^2\,dx-\sum_{i,j=1}^N\int 
	\phi x_j\partial^3_{iij}\bar \phi\,dx\\
	&=-2\int |\nabla \phi|^2\,dx-\sum_{i,j=1}^N\int 
	(\partial_j\phi x_j+\phi)\partial^2_{ii}\bar \phi\,dx\\
	&=-2\int |\nabla \phi|^2\,dx-\int 
	\Delta \phi\,x\cdot \nabla \bar \phi\,dx+N\int |\nabla \phi|^2\,dx,
	\end{align}
	and thus,
	\begin{align}\label{l2}
	\mbox{Re}\,\int \Delta \phi\,x\cdot \nabla \bar \phi\,dx=\left(\frac{N}{2}-1\right)\int |\nabla \phi|^2\,dx.
	\end{align}
	For the second term in the left hand side of \eqref{l1}, we also integrate by parts to deduce
	\begin{align}
	\int |x|^{-b}|\phi|^{2\sigma}\phi\,x\cdot \nabla \phi\,dx&=\sum_{j=1}^N\int |x|^{-b}|\phi|^{2\sigma}\phi x_j\partial_j\bar \phi\,dx=-\sum_{j=1}^{N}\int\partial_j(|x|^{-b}|\phi|^{2\sigma}\phi x_j)\bar \phi\,dx\\
	&=b\int |x|^{-b}|\phi|^{2\sigma+2}\,dx - 2\sigma\mbox{Re}\,\int |x|^{-b}|\phi|^{2\sigma}\bar \phi\, x\cdot\nabla \phi\,dx\\
	&\quad-N\int |x|^{-b}|\phi|^{2\sigma+2}\,dx-\int |x|^{-b}|\phi|^{2\sigma}\bar \phi\,x\cdot \nabla \phi\,dx,
	\end{align}
	which implies
	\begin{align}\label{l3}
	\mbox{Re} \int |x|^{-b}|\phi|^{2\sigma}\phi\,x\cdot \nabla \bar \phi\,dx=-\frac{N-b}{2\sigma+2}\int |x|^{-b}|\phi|^{2\sigma+2}\,dx.
	\end{align}
	The term in the right hand side of \eqref{l1} can be treated in the same way to obtain
	\begin{align}\label{l4}
	\mbox{Re}\,\int |\phi|^{\sigma_c-2}\phi\,x\cdot \nabla \bar \phi\,dx=-\frac{2-b}{2\sigma}\int |\phi|^{\sigma_c}\,dx.
	\end{align}
	From \eqref{l1}, \eqref{l2}, \eqref{l3} and \eqref{l4}, we get
	\begin{align}\label{l5}
	\frac{N-b}{2\sigma+2}\int |x|^{-b}|\phi|^{2\sigma+2}\,dx=\left(\frac{N}{2}-1\right)\int |\nabla \phi|^2\,dx+\frac{2-b}{2\sigma}\int |\phi|^{\sigma_c}\,dx.
	\end{align}
	Now, multiplying \eqref{groundstate} by $\phi$ and integrating, we obtain
	\begin{align}\label{l6}
	\int |x|^{-b}|\phi|^{2\sigma+2}\,dx=\int |\nabla \phi|^2\,dx+\int |\phi|^{\sigma_c}\,dx.
	\end{align}
	From the identities \eqref{l5} and \eqref{l6}, it is easy to deduce the relations \eqref{h1sc} and \eqref{epsc}. 
\end{proof}

Now, consider the functional space
\begin{align}
L^{2\sigma+2}_b(\Real^N)=\left\{f\in \mathcal{M}(\Real^N;\mathbb{C});\,\,\,\int |x|^{-b}|f|^{2\sigma+2}\,dx<\infty\right\},
\end{align}
where $\mathcal{M}(\Real^N;\mathbb{C})$ denotes the set of all measurable function $f:\Real^N \rightarrow \mathbb{C}$. In this space we define the norm
$$
\|f \|_{L^{2\sigma+2}_b}=\left(\int |x|^{-b}|f|^{2\sigma+2}\,dx\right)^{\frac{1}{2\sigma+2}}.
$$
In the next proposition we prove an useful compact embedding.
%
\begin{prop}\label{WSC}
Let $N\geq 1$, $0<b<2$, $\frac{2-b}{N}<\sigma<\frac{2-b}{N-2}\, \,(\frac{2-b}{N}<\sigma<\infty, \mbox{ if } N=1,2)$ and $\sigma_c=\frac{2N\sigma}{2-b}$. Then, the embedding  
\begin{align}
\dot H^{1}(\mathbb{R}^N)\cap L^{\sigma_c}(\mathbb{R}^N)\subset L^{2\sigma+2}_b(\mathbb{R}^N) 
\end{align}
is compact.
\end{prop}
\begin{proof}
Let $\{f_n\}_{n\in \mathbb{N}}$ be a bounded sequence in $\dot H^{1}(\mathbb{R}^N)\cap L^{\sigma_c}(\mathbb{R}^N) $. Then, there exists $f\in \dot H^{1}(\mathbb{R}^N)\cap L^{\sigma_c}(\mathbb{R}^N) $ such that, up to a subsequence, $f_n\rightharpoonup f$ in $\dot H^{1}(\mathbb{R}^N)\cap L^{\sigma_c}(\mathbb{R}^N)$, as $n\to \infty$. Defining $w_n=f_n-f$, we will show that 
\begin{align}
\int|x|^{-b}|w_n|^{2\sigma+2}\,dx\to 0, \quad \textnormal{as} \quad n\to \infty.
\end{align}
First, since $\{w_n\}_{n\in \mathbb{N}}$ is uniformly bounded in $\dot H^{1}(\mathbb{R}^N)\cap L^{\sigma_c}(\mathbb{R}^N) $, from the Gagliardo-Nirenberg inequality\footnote{For $N\geq 3$ it is enough to use the Sobolev embedding \eqref{SI} and interpolation.} \eqref{GNgeral2}, we get that 
\begin{align}\label{lmtaunif}
\{w_n\}_{n\in \mathbb{N}}\quad \textnormal{ is uniformly bounded in}\quad  L^p(\Real^N)\quad \textnormal{for all} \quad p\in (\sigma_c,2^*),
\end{align}
where 
\begin{align}
2^*=
\begin{cases}
\frac{2N}{N-2},\,\,N\geq 3\\
\infty,\,\,\,\,N=1,2.
\end{cases}
\end{align}

Moreover, given $\varepsilon>0$ for $R\geq \varepsilon^{-\frac{1}{b}}$ we have from \eqref{lmtaunif} that
\begin{align}\label{inteps}
\int_{\Real^N\backslash B(0,R)}|x|^{-b}|w_n|^{2\sigma+2}\,dx\leq {\varepsilon}\int_{\Real^N\backslash B(0,R)}|w_n|^{2\sigma+2}\leq C\varepsilon,
\end{align}
where $B(0,R)=\{ x\in \mathbb{R}^N;|x|\leq R\}$.


Now, we estimate the integral over the ball $B(0,R)$. 
Note that
\begin{equation}
\int_{B(0,R)}|\nabla w_n|^2\,dx\leq \int|\nabla w_n|^2\,dx.
\end{equation} 
Moreover, by H\"older's inequality and Sobolev embedding we get for $p\in(2,2^*)$
\begin{equation}
\int_{B(0,R)}|w_n|^2\,dx\leq C_{R,N} \| w_n\|_{L^{p}}^{2}\leq C_{R,N}\|\nabla w_n\|_{L^2}^2.
\end{equation}
and therefore $\{w_n\}_{n\in \mathbb{N}}$ is uniformly bounded in $ H^{1}(B(0,R))$. 



From the compact embedding on bounded domains $ H^1\left(B(0,R)\right)\subset L^p\left(B(0,R)\right)$ for $p\in (2,2^*)$ (see, for instance, \citet[Theorem 1.3.4]{cazenave}\footnote{In the case $N=1$ we have $H^1(B(0,R))\subset L^{\infty}(B(0,R))$. Since the embedding $L^{\infty}(B(0,R))\subset L^p(B(0,R))$ for all $p\geq1$ is continuous, we also have that the embedding $H^1(B(0,R))\subset L^p(B(0,R))$, for all $p\geq1$ is compact.}) and the fact that $w_n\rightharpoonup 0$ in $L^{\sigma_c}(\mathbb{R}^N)$, we deduce that 
\begin{align}\label{convstrong}
w_n\to 0\mbox{ in } L^p\left(B(0,R)\right)\quad \textnormal{for all} \quad p\in (2,2^*),
\end{align}
up to a subsequence. Again, since $\sigma<\frac{2-b}{N-2}$, we obtain $\frac{N-b}{N}-(\sigma+1)\frac{N-2}{N}>0$. Thus, we can choose $\gamma_2'\geq 1$ such that $\gamma_2'<\frac{N}{b}$ and $2<(2\sigma+2)\gamma_2<2^*$ (where $\gamma_2'$ is such that $\frac{1}{\gamma_2}+\frac{1}{\gamma_2'}=1)$. Hence, by H\"older's inequality we have
\begin{align}
\int_{B(0,R)}|x|^{-b}\left|w_n\right|^{2\sigma+2}\,dx&\leq \left(\int_{B(0,R)}|x|^{-b\gamma_2'}\,dx\right)^{\frac{1}{\gamma_2'}}\left(\int_{B(0,R)}\left|w_n\right|^{(2\sigma+2)\gamma_2}\right)^{\frac{1}{\gamma_2}}\\
&\leq \left(\int_{B(0,R)}|x|^{-b\gamma_2'}\,dx\right)^{\frac{1}{\gamma_2'}}\left(\int_{B(0,R)}|w_n|^{(2\sigma+2)\gamma_2}\,dx\right)^\frac{1}{\gamma_2}.
\end{align}
So, in view of \eqref{convstrong}, given $\varepsilon>0$ there exists $n_0$ such that for any $n\geq n_0$
\begin{align}
\int_{B(0,R)}|x|^{-b}\left|w_n\right|^{2\sigma+2}\,dx<{\varepsilon},
\end{align}
which completes the proof of Proposition \ref{WSC}.
\end{proof}
We now show Theorem \ref{GNU}, which characterizes the sharp constant for the Gagliardo-Nirenberg type inequality \eqref{GNcrit}.
\begin{proof}[Proof of Theorem \ref{GNU}]
	Given $f\in \dot H^{1}(\mathbb{R}^N)\cap L^{\sigma_c}(\mathbb{R}^N)$, define the Weinstein functional
	\begin{align}
	J(f)=\frac{\|\nabla f\|_{L^2}^2\|f\|_{L^{\sigma_c}}^{2\sigma}}{ \left\||\cdot|^{\frac{-b}{2\sigma+2}}f\right\|_{L^{2\sigma+2}}^{2\sigma+2}}.
	\end{align}
	We set $J=\displaystyle\inf_{f\in \dot H^{1}\cap L^{\sigma_c}, f\neq 0} J(f)$ and consider a minimizing sequence $\{f_n\}_{n\in \mathbb{N}}$. From \eqref{GNcrit}, we can deduce that $J>0$. Now rescale $\{f_n\}_{n\in \mathbb{N}}$ by setting
	\begin{align}
	g_n(x)=\mu_nf_n(\theta_n x),
	\end{align}
	with
	\begin{align}
	\mu_n=\frac{\|f_n\|_{L^{\sigma_c}}^{\frac{N-2}{2(1-s_c)}}}{\|\nabla f_n\|^{\frac{2-b}{2\sigma(1-s_c)}}_{L^2}}\,\,\,\,\,\mbox{ and }\,\,\,\,\theta_n=\left(\frac{\|f_n\|_{ L^{\sigma_c}}}{\|\nabla f_n\|_{L^2}}\right)^{\frac{1}{1-s_c}}.
	\end{align}
	A direct calculation implies
%
	
	\begin{align}
\|g_n\|_{L^{\sigma_c}}^{\sigma_c}=\mu_n^{\sigma_c}\int |f_n(\theta_n x)|^{\sigma_c}\,dx=\frac{\mu_n^{\sigma_c}}{\theta_n^{N}}\|f_n\|_{L^{\sigma_c}}^{\sigma_c}=1
\end{align}
and
\begin{align}
	\|\nabla g_n\|_{L^2}^2=\mu_n^2\int \theta_n^2|\nabla f_n(\theta_n x)|^2\,dx=\frac{\mu_n^2}{\theta_n^{N-2}}\|\nabla f_n\|_{L^2}^2=1,
	\end{align}
	that is, $\|g_n\|_{L^{\sigma_c}}=\|\nabla g_n\|_{L^2}=1.$ Moreover, since $J$ is invariant under this scaling, $\{g_n\}_{n\in \mathbb{N}}$ is also a minimizing sequence and bounded in $\dot H^{1}(\mathbb{R}^N)\cap L^{\sigma_c}(\mathbb{R}^N)$. Furthermore, there exists $g^{*}\in \dot H^{1}(\mathbb{R}^N)\cap L^{\sigma_c}(\mathbb{R}^N)$ such that, up to a subsequence, $g_n\rightharpoonup g^{*}$ weakly in $\dot H^{1}(\mathbb{R}^N)\cap L^{\sigma_c}(\mathbb{R}^N)$, and then
	\begin{align}
	\|g^*\|_{L^{\sigma_c}}\leq 
	1\,\,\,\,\,\,\mbox{ and }\,\,\,\,\,\,\|\nabla g^*\|_{L^2}\leq 1.
	\end{align}
	Thus, from Proposition \ref{WSC}, we have
	\begin{align}
	J\leq J(g^*)\leq \frac{1}{\left\||\cdot|^{\frac{-b}{2\sigma+2}}g^*\right\|_{L^{2\sigma+2}}^{2\sigma+2}}=\lim_{n\to\infty} \frac{1}{\left\||\cdot|^{\frac{-b}{2\sigma+2}}g_n\right\|_{L^{2\sigma+2}}^{2\sigma+2}}=J.
	\end{align}
	Consequently,
	\begin{align}
	J(g^*)=\frac{1}{\left\||\cdot|^{\frac{-b}{2\sigma+2}}g^*\right\|_{L^{2\sigma+2}}^{2\sigma+2}} \,\,\,\,\,\,\,\,\,\,\textnormal{ and }\,\,\,\,\,\,\,\,\,\,\|g^*\|_{L^{\sigma_c}}=\|\nabla g^*\|_{L^2}=1.
	\end{align}
	In particular, $g^*\neq 0$ and is a minimizer for the functional $J$. Moreover, $g^*$ is a solution for the Euler-Lagrange equation
	\begin{align}\label{el}
	\frac{d}{ds}\Big|_{s=0} J(g_s)=0,
	\end{align}
	where $g_s=g^*+s\varphi$ for $\varphi\in \mathcal S(\Real^N)$.
	On the other hand, computing the Fr\'echet derivative, we deduce
%
$$
\frac{d}{ds}\Big|_{s=0}\left\||\cdot|^ {-\frac{b}{2\sigma+2}}g_s\right\|_{L^{2\sigma+2}}^{2\sigma+2}=(2\sigma+2)\mbox{Re}\,\left\langle|x|^{-b}|g^*|^{2\sigma}g^*,\varphi\right\rangle,
$$
$$
\frac{d}{ds}\Big|_{s=0}\|\nabla g_s\|_{L^2}^2=2\mbox{Re}\,\langle\nabla g^*, \nabla \varphi\rangle=2\mbox{Re}\,\langle -\Delta g^*,\varphi \rangle
$$
and
$$
\frac{d}{ds}\Big|_{s=0}\|g_s\|_{L^{\sigma_c}}^{2\sigma}=\frac{2\sigma}{\sigma_c}\|g^*\|_{L^{\sigma_c}}^{2\sigma-\sigma_c}\frac{d}{ds}\Big|_{s=0}\|g_s\|_{L^{\sigma_c}}^{\sigma_c}=2\sigma\|g^*\|_{L^{\sigma_c}}^{2\sigma-\sigma_c}\mbox{Re}\,\left\langle|g^*|^{\sigma_c-2}g^*,\varphi\right\rangle.
$$
	From \eqref{el}, $\|v^*\|_{L^{\sigma_c}}=\|\nabla v^*\|_{L^2}^2=1$ and $$J(f)\left\||\cdot|^ {-\frac{b}{2\sigma+2}}f\right\|_{L^{2\sigma+2}}^{2\sigma+2}=\|\nabla f\|_{L^2}^2\|f\|_{L^{\sigma_c}}^{2\sigma},$$ we have
	\begin{align}\label{eqell}
	\mbox{Re}\,\left\langle -\Delta g^*-(\sigma+1)J|x|^{-b}|g^*|^{2\sigma}g^*+\sigma|g^*|^{\sigma_c-2}g^*,\varphi\right\rangle=0,
	\end{align}
	where $\langle\cdot,\,\cdot\rangle$ is the inner product in $L^2$.
	Taking $i\varphi$ instead of $\varphi$ in \eqref{eqell} and using that Re\,$(iz)=-$Im\,$(z)$, we get
	\begin{align}
	\mbox{Im}\,\left\langle -\Delta v^*-(\sigma+1)J|x|^{-b}|v^*|^{2\sigma}v^*+\sigma|v^*|^{\sigma_c-2}v^*,\varphi\right\rangle=0.
	\end{align}
	Consequently, $g^*$ is a solution for elliptic equation
	\begin{align}
	\Delta g^*+(\sigma+1)J|x|^{-b}|g^*|^{2\sigma}g^*=\sigma|g^*|^{\sigma_c-2}g^*.
	\end{align}
	Now, we take $V$ defined by $g^*(x)=\alpha V(\beta x)$ with
	\begin{align}\label{alfa}
	\alpha=\left[\frac{(\sigma+1)J}{\sigma^{\frac{2-b}{2}}}\right]^{\frac{2N}{(2-b)[N(\sigma_c-2)-2\sigma_c]}}\,\,\,\,\,\,\,\,\,\,\,\,\mbox{ and } \,\,\,\,\,\,\,\,\,\,\,\,\beta=\sigma^{\frac{1}{2}}\left[\frac{(\sigma+1)J}{\sigma^{\frac{2-b}{2}}}\right]^{\frac{N(\sigma_c-2)}{(2-b)[N(\sigma_c-2)-2\sigma_c]}}
	\end{align}
	so that $V$ is a solution of \eqref{elptcpc} and
	\begin{align}\label{Vc}
	\|V\|_{L^{\sigma_c}}^{\sigma_c}=\frac{\beta^N}{\alpha^{\sigma_c}}\|g^*\|_{L^{\sigma_c}}^{\sigma_c}=\frac{\beta^N}{\alpha^{\sigma_c}}=[(\sigma+1)J]^{\frac{N}{2-b}}.
	\end{align}
	Note that, this implies
	\begin{align}\label{J}
	J=\frac{\|V\|_{L^{\sigma_c}}^{2\sigma}}{\sigma+1}
	\end{align}
	and by the definition of $J$, we have for all $f\in\dot H^1(\Real^N)\cap L^{\sigma_c}(\Real^N)$
	\begin{align}
	\frac{\|V\|_{L^{\sigma_c}}^{2\sigma}}{\sigma+1}=J\leq \frac{\|\nabla f\|_{L^2}^2\|f\|_{L^{\sigma_c}}^{2\sigma}}{ \left\||\cdot|^{\frac{-b}{2\sigma+2}}f\right\|_{L^{2\sigma+2}}^{2\sigma+2}},
	\end{align}
	which implies \eqref{GNsc}. 
	On the other hand, by Lemma \ref{pohozaev}, if $\phi$ is a solution of the elliptic equation \eqref{groundstate}, then $J(\phi)$ is given by
	\begin{align}\label{jellptic}
	J(\phi)=\frac{\|\nabla \phi\|_{L^2}^2\|\phi\|_{L^{\sigma_c}}^{2\sigma}}{\left\||\cdot |^{-\frac{b}{2\sigma+2}}\phi\right\|_{L^{2\sigma+2}}^{2\sigma+2}}=\frac{\frac{1}{\sigma}\|\phi\|_{L^{\sigma_c}}^{2\sigma+\sigma_c}}{\frac{\sigma+1}{\sigma}\|\phi\|_{L^{\sigma_c}}^{\sigma_c}}=\frac{\|\phi\|_{L^{\sigma_c}}^{2\sigma}}{\sigma+1}.
	\end{align}
	Therefore, since $J(V)=\displaystyle\min_{f\in \dot H^1\cap L^{\sigma_c}, f\neq 0}J(f)$, we have that $V$ is a solution of \eqref{elptcpc} with minimal $L^{\sigma_c}(\Real^N)$-norm.
\end{proof}
As a consequence of the Theorem \ref{GNU}, we obtain the following global well-posedness result.
\begin{proof}[Proof of Theorem \ref{global}]
	By the Gagliardo-Nirenberg inequality in Theorem \ref{GNU} and energy conservation, we get
	\begin{align}
	E(u_0)=E(u(t))=&\frac{1}{2}\|\nabla u(t)\|_{L^2}^2-\frac{1}{2\sigma+2}\left\||\cdot|^{-\frac{b}{2\sigma+2}}u(t)\right\|_{L^{2\sigma+2}}^{2\sigma+2}\\ \geq& \frac{1}{2}\|\nabla u(t)\|_{L^2}^2\left(1-\frac{\|u(t)\|_{L^{\sigma_c}}^{2\sigma}}{\|V\|_{L^{\sigma_c}}^{2\sigma}}\right).
	\end{align}
	Since $\dot H^{s_c}(\mathbb{R}^N)\subset L^{\sigma_c}(\mathbb{R}^N)$, from the assumption $\sup_t\|u(t)\|_{\dot H^{s_c}}<\|V\|_{L^{\sigma_c}}$, we deduce that $\|\nabla u(t)\|_{L^2}$ is bounded for all $t\in [0,T^*)$. On the other hand, since $\|u(t)\|_{\dot H^{s_c}\cap \dot H^1}=\|u(t)\|_{\dot H^{s_c}}+\|u(t)\|_{\dot H^{1}}$, by the blow-up alternative (see Corollary \ref{Blowalt}) if $T^*<\infty$, then $\lim_{t\uparrow T^*}\|\nabla u(t)\|_{L^2}=\infty$, which is a contradiction. Consequently, $u$ is a global solution for \eqref{PVI}.
\end{proof}

\begin{rem}\label{RemGWP} 
It is possible to replace the assumption $\sup_{t\in [0,T^*)}\|u(t)\|_{\dot H^{s_c}}<\|V\|_{L^{\sigma_c}}$ in the statement of Theorem \ref{global} by $\sup_{t\in [0,T^*)}\|u(t)\|_{L^{\sigma_c}}<\|V\|_{L^{\sigma_c}}$ and \eqref{condbounded}. Indeed, from $\sup_{t\in [0,T^*)}\|u(t)\|_{L^{\sigma_c}}<\|V\|_{L^{\sigma_c}}$ we deduce that $\|\nabla u(t)\|_{L^2}$ is bounded for all $t\in [0,T^*)$. Moreover, the assumption \eqref{condbounded} and Corollary \ref{Blowalt} implies $\lim_{t\uparrow T^*}\|\nabla u(t)\|_{L^2}=\infty$, achieving the contradiction.
\end{rem}

\section{Critical norm concentration}\label{CNC}

In this section, we prove our main result about $L^{\sigma_c}$-norm concentration in the intercritical regime for finite time blow-up solution. 
\begin{proof}[Proof of Theorem \ref{concentration}]
Let $\{t_n\}_{n\in \mathbb{N}}$ be an arbitrary time sequence such that $t_n\uparrow T^*$, as $n\to\infty$. Define
\begin{align}
\rho_{n}=\left(\frac{1}{\|\nabla u(t_n)\|_{L^2}}\right)^{\frac{1}{1-s_c}}\,\,\,\,\,\,\,\,\mbox{ and }\,\,\,\,\,\,\,v_n(x)=\rho_n^{\frac{2-b}{2\sigma}}u(\rho_n x,t_n).
\end{align}
In this case, for all $n\in\mathbb{N}$, we get
$$
\|v_n\|_{\dot H^{s_c}}=\|u(t_n)\|_{\dot H^{s_c}}<\infty,
$$
$$
\|\nabla v_n\|_{L^2}^2=\int |\nabla v_n|^2\,dx=\rho_n^{\frac{2(2-b)}{2\sigma}+2-N}\|\nabla u(t_n)\|_{L^2}^{2}=\rho_n^{2(1-s_c)}\|\nabla u(t_n)\|_{L^2}^{2}=1
$$
and
$$
E(v_n)=\frac{1}{2}\|\nabla v_n\|_{L^2}^{2}-\frac{1}{2\sigma+2}\int |x|^{-b}|v_n|^{2\sigma+2}\,dx=\rho_n^{2(1-s_c)}E(u_0).
$$
So $\{v_n\}_{n\in \mathbb{N}}$ is a bounded sequence in $\dot H^{s_c}(\mathbb{R}^N)\cap \dot H^{1}(\mathbb{R}^N)$ and since $\rho_n\to 0$ when $n\to \infty$, we have  
$$\displaystyle\lim_{n\to \infty}E(v_n)=0.$$ 
Thus, there exists $v^*\in \dot H^{s_c}(\mathbb{R}^N)\cap \dot H^{1}(\mathbb{R}^N)$ such that, up to a subsequence, $v_n\rightharpoonup v^*$ in $\dot H^{s_c}(\mathbb{R}^N)\cap \dot H^{1}(\mathbb{R}^N)$, as $n\to\infty$. Moreover since $\dot H^{s_c}(\mathbb{R}^N)\subset L^{\sigma_c}(\mathbb{R}^N)$ we have
\begin{align}\label{liminffrac}
\|\nabla v^*\|_{L^2}\leq \liminf_{n\to \infty}\|\nabla v_n\|_{L^2}\,\,\,\,\mbox{ and }\,\,\,\,\|v^*\|_{L^{\sigma_c}}\leq \liminf_{n\to \infty}\|v_n\|_{L^{\sigma_c }}.
\end{align}
In addition, by Proposition \ref{WSC}
\begin{align}\label{convgweak}
\lim_{n\to\infty}\left\||\cdot|^{-\frac{b}{2\sigma+2}}v_n\right\|_{L^{2\sigma+2}}^{2\sigma+2}=\left\||\cdot|^{-\frac{b}{2\sigma+2}}v^* \right\|_{L^{2\sigma+2}}^{2\sigma+2}.
\end{align}
Hence, from the sharp Gagliardo-Nirenberg inequality in Theorem \ref{GNU}, \eqref{liminffrac} and \eqref{convgweak}, we get
\begin{align}
0=\liminf_{n\to\infty}E(v_n)\geq \frac{1}{2}\|\nabla v^*\|_{L^2}^{2}\left(1-\frac{\|v^*\|_{L^{\sigma_c}}^{2\sigma}}{\|V\|_{L^{\sigma_c}}}\right),
\end{align}
which implies $\|v^*\|_{L^{\sigma_c}}\geq \|V\|_{L^{\sigma_c}}$. Consequently, for every $R>0$,
\begin{align}
\liminf_{n\to\infty}\int_{|y|\leq\rho_nR}|u(y,t_n)|^{\sigma_c}\,dy=&\liminf_{n\to\infty}\int_{|x|\leq R}\rho_n^{\frac{\sigma_c(2-b)}{2\sigma}}|u(\rho_n x,t_n)|^{\sigma_c}\,dx\\=&\liminf_{n\to\infty}\int_{|x|\leq R}|v_n(x)|^{\sigma_c}\,dx\geq \int_{|x|\leq R}|v^*|^{\sigma_c}\,dx,
\end{align}
where we have used the weak convergence $v_n\rightharpoonup v^*$ in $L^{\sigma_c}(\Real^N)$ in the last inequality.
Using the assumption $\lambda(t_n)/\rho_n\to \infty$ as $n\to\infty$, we obtain
\begin{align}
\liminf_{n\to\infty}\int_{|x|\leq \lambda(t_n)}|u(x,t_n)|^{\sigma_c}\,dx\geq \int_{|x|\leq R}|v^*|^{\sigma_c}\,dx,
\end{align}
for all $R>0$, which gives
\begin{align}
\liminf_{n\to\infty}\int_{|x|\leq \lambda(t_n)}|u(x,t_n)|^{\sigma_c}\,dx\geq \|V\|^{\sigma_c}_{L^{\sigma_c}}.
\end{align}
Since $\{t_n\}_{n\in \mathbb{N}}$ is arbitrary, we deduce that
\begin{align}
\liminf_{t\to T^*}\int_{|x|\leq \lambda(t)}|u(x,t)|^{\sigma_c}\,dx\geq \|V\|^{\sigma_c}_{L^{\sigma_c}},
\end{align}
which completes the proof.
\end{proof} 
\section{A remark on another concentration result for the INLS equation}\label{ACR}
Here, we present some results which can be obtained in a  similar manner to those established in the two previous sections. First, Proposition \ref{WSC} allows us to obtain an alternative Gagliardo-Niremberg type inequality.
\begin{thm}\label{GNUsc}
	Let $N\geq 1$, $0<b<2$, $\frac{2-b}{N}<\sigma<\frac{2-b}{N-2}$ ($\frac{2-b}{N}<\sigma<\infty$, if $N=1,2$) and $s_c=\frac{N}{2}-\frac{2-b}{2\sigma}$, then the following Gagliardo-Nirenberg inequality holds for all $f\in \dot H^{s_c}\cap\dot H^1$
	\begin{align}\label{GNisc}
	\int_{\Real^N} |x|^{-b} |f(x)|^{2\sigma+2}\,dx\leq \frac{\sigma+1}{\|W\|_{\dot H^{s_c}}^{2\sigma}}\|\nabla f\|_{L^2}^2\|f\|_{\dot H^{s_c}}^{2\sigma},
	\end{align}
	where $W$ is a solution to the elliptic equation, 
	\begin{align}\label{elptcsc1}
	\Delta W+|x|^{-b}|W|^{2\sigma}W-(-\Delta)^{s_c}W=0
	\end{align}
with minimal $\dot H^{s_c}$-norm.
\end{thm}
\begin{proof}[Sketch of the proof of Theorem \ref{GNUsc}]
	\textbf{Step 1.} Define the Weinstein functional 
	\begin{align}
	J(f)=\frac{\|\nabla f\|_{L^2}^{2}\|f\|_{\dot H^{s_c}}^{2\sigma}}{\left\||\cdot|^{-\frac{b}{2\sigma+2}}f\right\|_{L^{2\sigma+2}}^{2\sigma+2}}
	\end{align}
	and the number $$J=\displaystyle\inf_{f\in \dot H^{s_c}\cap\dot H^1, f\neq 0}J(f).$$
	\textbf{Step 2.}  Let $\{f_n\}_{n\in \mathbb{N}}$ be a minimizing sequence in $\dot H^{s_c}(\mathbb{R}^N)\cap \dot H^1(\mathbb{R}^N)$ and consider the sequence $w_n(x)=\mu_nf_n(\theta_nx)$ with
	\begin{align}
	\mu_n=\frac{\|f_n\|_{\dot H^{s_c}}^{\frac{N-2}{2(1-s_c)}}}{\|\nabla f_n\|_{L^2}^{\frac{2-b}{2\sigma(1-s_c)}}}\,\,\,\,\mbox{ and }\,\,\,\,\theta_n=\left(\frac{\|f_n\|_{\dot H^{s_c}}}{\|\nabla f_n\|_{L^{2}}}\right)^{\frac{1}{1-s_c}}.
	\end{align}
	Thus, $\|w_n\|_{\dot H^{s_c}}=\|\nabla w_n\|_{L^2}=1$ and $\{w_n\}_{n\in \mathbb{N}}$ is a minimizing sequence in $\dot H^{s_c}(\mathbb{R}^N)\cap \dot H^1(\mathbb{R}^N)$.\\
	\textbf{Step 3.} There exists $w^*\in \dot H^{s_c}(\mathbb{R}^N)\cap \dot H^1(\mathbb{R}^N)$ such that $w_n\rightharpoonup w^*$ in $\dot H^{s_c}(\mathbb{R}^N)\cap \dot H^1(\mathbb{R}^N)$, as $n\to \infty$. Moreover,
	$$\|w^*\|_{\dot H^{s_c}}\leq 1\,\,\,\,\mbox{ and }\,\,\,\,\|\nabla w^*\|_{L^2}\leq 1,$$
	and from Proposition \ref{WSC} it follows that $J(w^*)=J.$\\
	\textbf{Step 4.} Since $w^*$ is a solution for Euler-Lagrange equation
	$$\frac{d}{ds}\Big|_{s=0}J(w_s)=0,$$
	where $w_s=w^*+s\varphi$ for $\varphi\in \mathcal S(\Real^N)$ we deduce that
	$$\langle -\Delta w^*-(\sigma+1)J|x|^{-b}|w^*|^{2\sigma}w^*+\sigma(-\Delta)^{s_c}w^*,\varphi\rangle=0,$$
	where $(-\Delta)^{s_c}w^*$ satisfies $\langle (-\Delta)^{s_c}w^*,\varphi\rangle=\langle D^{s_c}w^*,D^{s_c}\varphi\rangle$ for all $\varphi\in \mathcal S(\Real^N)$.\\
	\textbf{Step 5.} If $W$ is given by $w^*(x)=\alpha W(\beta x)$ with
	\begin{align}
	\alpha=\left[\frac{\sigma^{\frac{2-b}{2(1-s_c)}}}{(\sigma+1)J}\right]^{\frac{1}{2\sigma}}\,\,\,\,\,\mbox{ and }\,\,\,\,\,\beta =\sigma^{\frac{1}{2(1-s_c)}},
	\end{align}
	then 
	$$\|W\|_{\dot H^{s_c}}=\frac{\beta^{\frac{2-b}{2\sigma}}}{\alpha}\|w^*\|_{\dot H^{s_c}}=[(\sigma+1)J]^{\frac{1}{2\sigma}}$$
	and $W$ is a solution (in the weak sense) to the elliptic equation
	\begin{align}
	\Delta W+|x|^{-b}|W|^{2\sigma}W=(-\Delta)^{s_c}W,
	\end{align}
	with minimal $\dot H^{s_c}$-norm.\\
	\textbf{Step 6.} Therefore, for all $f\in \dot H^{s_c}(\mathbb{R}^N)\cap \dot H^1(\mathbb{R}^N)$
	$$\int|x|^{-b}|f(x)|^{2\sigma+2}\leq \frac{1}{J}\|\nabla f\|_{L^2}^2\|u\|_{\dot H^{s_c}}^{2\sigma}=\frac{\sigma+1}{\|W\|_{\dot H^{s_c}}}\|\nabla f\|_{L^2}^2\|f\|_{\dot H^{s_c}}^{2\sigma}$$
	which completes the proof.
\end{proof}
Next, with this new Gagliardo-Nirenberg type inequality in hand we can prove a variant global well-posedness and concentration results. We will omit the proofs as they are completely analogous to the proofs of Theorem \ref{global} and Theorem \ref{concentration}.
\begin{thm}\label{global2}
	Let $N\geq3$, $0<b<\min\{\frac{N}{2},2\}$, $\frac{2-b}{N}<\sigma<\frac{2-b}{N-2}$, $s_c=\frac{N}{2}-\frac{2-b}{2\sigma}$ and $\sigma_c=\frac{2N\sigma}{2-b}$. For $u_0\in \dot H^{s_c}(\mathbb{R}^N)\cap \dot H^1(\mathbb{R}^N)$, let $u(t)$ be the corresponding solution to \eqref{PVI} given by Theorem \ref{LWP} and $T^*>0$ the maximal time of existence. Suppose that $\sup_{t\in [0,T^*)}\|u(t)\|_{\dot H^{s_c}}<\|W\|_{\dot H^{s_c}}$, where $W$ is a solution of the elliptic equation \eqref{elptcsc1} with minimal $\dot H^{s_c}$-norm. Then $u(t)$ exists globally in the time.
\end{thm}
From Proposition \ref{WSC} and Theorem \ref{GNUsc}, we also obtain the following $\dot H^{s_c}$-norm concentration.
\begin{thm}\label{concentrationsc} Let $N\geq3$, $0<b<\min\{\frac{N}{2},2\}$ and $\frac{2-b}{N}<\sigma<\frac{2-b}{N-2}$. For $u_0\in \dot{H}^{s_c}\cap\dot{H}^1$, let $u(t)$ be the corresponding solution to \eqref{PVI} given by Theorem \ref{LWP} and assume that it blows up in finite time $T^*>0$ satisfying \eqref{condbounded}.
	If $\lambda (t)>0$ is such that
	\begin{align}
	\lambda(t)\|\nabla u(t)\|_{L^2}^{\frac{1}{1-s_c}}\to  \infty, \,\,\textit{ as }\, t\to T^*,
	\end{align}
	then,
	\begin{align}
	\liminf_{t\to T^*} \int_{|x|\leq \lambda(t)}|D^{s_c}u(x,t)|^{2}\,dx\geq \|W\|^{2}_{\dot H^{s_c}}
	\end{align}
	where $W$ is a minimal $\dot H^{s_c}$-norm solution to elliptic equation \eqref{elptcsc1}.
\end{thm}

\vspace{0.5cm}
\noindent 
\textbf{Acknowledgments.} M.C. was partially supported by Coordena\c{c}\~ao de Aperfei\c{c}oamento de Pessoal de N\'ivel Superior - CAPES. L.G.F. was partially supported by Coordena\c{c}\~ao de Aperfei\c{c}oamento de Pessoal de N\'ivel Superior - CAPES, Conselho Nacional de Desenvolvimento Cient\'ifico e Tecnol\'ogico - CNPq and Funda\c{c}\~ao de Amparo a Pesquisa do Estado de Minas Gerais - Fapemig/Brazil. 

\appendix
\section{Appendix}\label{appA}
Let $N\geq3$, $0<b<2$, $\frac{2-b}{N}<\sigma<\frac{2-b}{N-2}$ and $s_c=\frac{N}{2}-\frac{2-b}{2\sigma}$. Given $(q,p)$ $L^2$-admissible and $(a,r)$ $\dot H^{s_c}$-admissible, we define the space (recall definition \eqref{StrNorm}) 
$$
X= \left( \bigcap_{(q,p)\in \mathcal{A}_0}L^q\left([-T,T];\dot{H}^{s_c,p}\cap \dot{H}^{1,p}\right)\right)\bigcap \left( \bigcap_{(a,r)\in \mathcal{A}_{s_c}}L^a\left([-T,T]; L^r \right)\right)
$$
equipped with the norm
$$
\|u\|_T=\|\nabla u\|_{S\left(L^2;I\right)}+\|D^{s_c} u\|_{S\left(L^2;I\right)}+\|u\|_{S\left(\dot
	H^{s_c};I\right)},
$$
where $I=[-T,T]$.
For $m,T>0$, let $S(m,T)$ be the set
\begin{align}
S(m,T)=\left\{u\in X;\,\|u\|_{T}\leq m\right\}
\end{align}
and $d_T$ the metric in $X$ given by
\begin{align}
d_T(u,v)=\|u-v\|_{S\left(\dot H^{s_c};I\right)}.
\end{align}

We use a similar argument to the one in \citet[Theorem 4.4.1]{cazenave} to show the following result.
\begin{lemma}
	$\left(S(m,T),d_T\right)$ is a complete metric space.
\end{lemma}
\begin{proof}
	
	Since $S(m,T)\subset X$ and $X$ is a Banach space, it suffices to show that $S(m,T)$ with the metric $d_T$ is closed in $S(\dot H^{s_c};I)$. For this end, let $\{u_n\}_{n\in \mathbb{N}}\subset S(m,T)$ and $u\in S(\dot H^{s_c};I)$ such that $d_T(u_n,u)\to0$, as $n\to \infty$.
	This means that $u_n\to u$ in $L^{a}_TL^r_x$ as $n\to\infty$ for all $(a,r)$  $\dot H^{s_c}$-admissible and, in particular,  
	\begin{equation}\label{Conv-unt}
	u_n(t)\to u(t) \,\, \textnormal{in} \,\, L^r_x(\mathbb{R}^N) \,\, \textnormal{as} \,\, n\to\infty \,\, \textnormal{for a.a.} \,\, t\in [-T,T].
	\end{equation}
 Moreover, since
\begin{align}
\|u\|_{L^a_TL^r_x}\leq \|u_n-u\|_{L^a_TL^r_x}+ \|u_n\|_{L^a_TL^r_x},\,\,\,\,\,\mbox{ for all } (a,r) \,\,\dot H^{s_c}\mbox{-admissible},
\end{align} 
and taking $n\to \infty$, we have 
\begin{align}\label{Hsc}
\|u\|_{L^a_TL^r_x}\leq \liminf_{n\to\infty}\|u_n\|_{L^a_TL^r_x}.
\end{align}

Now, let $(q,p)\neq (\infty,2)$ be a $L^2$-admissible pair. Since $\{u_n\}_{n\in \mathbb N}$ is a bounded sequence in $ L^q_T\dot H^{1,p}$, we have
\begin{align}
u\in L^q\left([-T,T]; \dot H^{1,p}\right),
\end{align} 
$u_n\rightharpoonup u$ in $L^q_T\dot H^{1,p}$, as $n\to \infty$ and
	 	\begin{align}\label{qp}
	 	 \|\nabla u\|_{L^q_TL^{p}_x}\leq \liminf_{n\to \infty} \|\nabla u_n\|_{L^q_TL^{p}_x}.
	 	\end{align}
In the same way, we also deduce
\begin{align}\label{qp2}
	 	\|D^{s_c} u\|_{L^q_TL^{p}_x}\leq \liminf_{n\to \infty} \|D^{s_c} u_n\|_{L^q_TL^{p}_x}.
	 	\end{align}

	 	We now consider the pair $(q,p)=(\infty,2)$. From the Sobolev embedding \eqref{SEsc}, we have that $\dot H^{s_c}(\mathbb{R}^N)\cap\dot H^1(\mathbb{R}^N) \subset \dot H^{s_c}(\mathbb{R}^N)\subset  L^{\sigma_c}(\mathbb{R}^N)$, for $\sigma_c=\frac{2N}{N-2s_c}$. Since the pair $(\infty,\frac{2N}{N-2s_c})$ is $ \dot H^{s_c}$- admissible, we have that \eqref{Conv-unt} holds for $r=\sigma_c$. In addition, by the definition of $S(m,T)$, we have that $\{u_n\}_{n\in \mathbb{N}}$ is a bounded sequence in $L^{\infty}\left([-T,T];\dot H^{s_c}\cap \dot H^1\right)$. Thus, by \citet[Theorem 1.2.5]{cazenave}, it follows that 
	 	\begin{align}
	 	u\in L^{\infty}\left([-T,T];\dot H^{s_c}\cap \dot H^1\right)
	 	\end{align}
	 	and
	\begin{align}\label{L2}
	\|\nabla u\|_{L^\infty_TL^2_x}+\|D^{s_c} u\|_{L^\infty_TL^2_x}\leq \liminf_{n\to \infty} \left(\|\nabla u_n\|_{L^\infty_TL^2_x}+\|D^{s_c} u_n\|_{L^\infty_TL^2_x}\right).
	\end{align}
	Hence, using \eqref{Hsc}, \eqref{qp}, \eqref{qp2} and \eqref{L2} for all $(q,p)$ $L^2$-admissible and $(a,r)$ $\dot H^{s_c}$-admissible
	\begin{align}
	\|\nabla u\|_{L^q_TL^p_x}+\|D^{s_c} u\|_{L^q_TL^p_x}+\| u\|_{L^a_TL^r_x}\leq \liminf_{n\to \infty}\left(	\|\nabla u_n\|_{L^q_TL^p_x}+\|D^{s_c} u_n\|_{L^q_TL^p_x}+\| u_n\|_{L^a_TL^r_x}\right)\leq m.
	\end{align}
	Therefore, 
	\begin{align}
	\|\nabla u\|_{S(L^2;[-T,T])}+\|D^{s_c}u\|_{S(L^2;[-T,T])}+\|u\|_{S(\dot H^{s_c};[-T,T])}\leq  m,
	\end{align}
	and thus, $u\in S(m,T)$, which completes the proof.
\end{proof}


\newcommand{\Addresses}{{
		\bigskip
		\footnotesize
		
		MYKAEL A. CARDOSO, \textsc{Department of Mathematics, UFMG, Brazil;}
		\textsc{Department of Mathematics, UFPI, Brazil}\par\nopagebreak
		\textit{E-mail address:} \texttt{mykael@ufpi.edu.br}
		
		\medskip
		
		LUIZ G. FARAH, \textsc{Department of Mathematics, UFMG, Brazil}\par\nopagebreak
		\textit{E-mail address:} \texttt{farah@mat.ufmg.br}

		\medskip
		 CARLOS M. GUZM\'AN, \textsc{Department of Mathematics, UFF, Brazil;}\par\nopagebreak
		 \textit{E-mail address:} \texttt{carlos.guz.j@gmail.com}

}}
\setlength{\parskip}{0pt}
\Addresses

\end{document}